\documentclass[headings]{article}

\usepackage[utf8]{inputenc}
\usepackage[french]{babel}

\usepackage{amsthm}

\usepackage{fancyhdr}
\usepackage{microtype}

\usepackage{amsmath}
\usepackage{amssymb}
\usepackage{frcursive}
\usepackage{mathrsfs}
\usepackage{latexsym}
\usepackage{stmaryrd}
\usepackage{dsfont} 
\usepackage[all]{xy} 

\usepackage{boxedminipage}

\usepackage{listings}

\usepackage{minitoc}

\usepackage{ifpdf}

\ifpdf
\usepackage[pdftex]{graphicx}
\else
\usepackage{graphicx}
\fi

\title{Une expression spectrale pour une certaine intégrale orbitale tordue}
\author{Joël Cohen}

\newtheorem{thm}{Théorème}[section]
\newtheorem{defn}[thm]{Définition}

\newtheorem{prop}[thm]{Proposition}
\newtheorem{lemme}[thm]{Lemme}
\newtheorem{rem}[thm]{Remarque}
\newtheorem{ex}[thm]{Exemple}
\newtheorem{cor}[thm]{Corollaire}

\newtheorem{conj}[thm]{Conjecture}

\newcommand{\N}{\mathbb{N}}
\newcommand{\Z}{\mathbb{Z}}
\newcommand{\Q}{\mathbb{Q}}
\newcommand{\R}{\mathbb{R}}
\newcommand{\C}{\mathbb{C}}
\newcommand{\F}{\mathbb{F}}

\newcommand{\g}{\lie{g}}
\newcommand{\h}{\lie{h}}

\newcommand{\Orb}{\mathcal{O}}
\newcommand{\OrbInt}{J}
\newcommand{\Hecke}{\mathcal{H}}
\newcommand{\hecke}{\mathcal{C}_{c}^{\infty}}

\newcommand{\Ind}{\operatorname{Ind}}

\newcommand{\Stab}{\operatorname{Stab}}
\newcommand{\Irr}{\operatorname{Irr}}
\newcommand{\GL}{\operatorname{GL}}
\newcommand{\gl}{\lie{gl}}

\newcommand{\SO}{\operatorname{SO}}
\newcommand{\so}{\lie{so}}
\newcommand{\Sp}{\operatorname{Sp}}
\newcommand{\Hom}{\operatorname{Hom}}

\newcommand{\Aut}{\operatorname{Aut}}

\newcommand{\Supp}{\operatorname{Supp}}

\newcommand{\Rep}{\mathcal{R}} 
\newcommand{\D}{\operatorname{d}} 

\newcommand{\st}{\operatorname{st}}
\newcommand{\reg}{\textnormal{reg}}
\newcommand{\temp}{\operatorname{temp}}

\newcommand{\pars}[1]{\ensuremath{\left( #1 \right)}} 
\newcommand{\braces}[1]{\ensuremath{\left\{ #1 \right\}}} 
\newcommand{\set}[1]{\ensuremath{\left\{ #1 \right\}}} 
\newcommand{\scal}[1]{\ensuremath{\left\langle #1 \right\rangle}} 
\newcommand{\ints}[1]{\ensuremath{\left[\!\left[ #1 \right]\!\right]}} 
\newcommand{\mat}[1]{\ensuremath{\pars{\begin{matrix} #1 \end{matrix}}}} 
\newcommand{\alg}[1]{\boldsymbol{#1}} 
\newcommand{\lie}[1]{\ensuremath{\mathfrak{#1}}} 

\def\revdots{\mathinner{\mkern1mu\raise1pt\vbox{\kern7pt\hbox{.}}\mkern2mu\raise4pt\hbox{.}\mkern2mu \raise7pt\hbox{.}\mkern1mu}} 

\newcommand{\Tr}{\operatorname{Tr}}
\newcommand{\Id}{\operatorname{Id}}

\newcommand{\Ad}{\operatorname{Ad}}
\newcommand{\Inn}{\operatorname{Inn}}

\newcommand{\diag}{\operatorname{diag}}

\newcommand{\Gal}{\operatorname{Gal}}

\newcommand{\transfertZero}{\mathcal{C}_0} 
\newcommand{\adherenceIntOrb}{\mathcal{C}_0} 
\newcommand{\distr}{\mathcal{D}} 

\begin{document}

\ifpdf
\DeclareGraphicsExtensions{.pdf, .jpg, .tif}
\else
\DeclareGraphicsExtensions{.eps, .jpg}
\fi

\maketitle

\tableofcontents

\ \\
\ \\


Soit $F$ un corps $p$-adique, $G = \GL_{2n}(F)$ et $\theta_0$ l'automorphisme extérieur de $G$ qui préserve une paire de Borel épinglée. On considère l'ensemble $\widetilde{G} = G \theta_0$ sur lequel $G$ agit par conjugaison et l'intégrale orbitale $J_{\widetilde{G}}(\theta_0, f)$ en $\theta_0$. On donne une expression spectrale type Plancherel-Harish-Chandra pour cette intégrale orbitale, c'est-à-dire comme une intégrale sur les représentations irréductibles tempérées auto-duales de $G$ dites ``symplectiques'' (c'est-à-dire dont le paramètre de Langlands se factorise par $\Sp_{2n}(\C)$). La méthode utilise le transfert endoscopique vers $\SO_{2n+1}$. On démontre au passage que la mesure de Plancherel est constante sur les $L$-paquets.

Ce travail fait partie de ma thèse de doctorat préparée sous la direction de Volker Heiermann. Cette thèse a été financée par l'Agence Nationale de la Recherche dans le cadre du projet ANR blanc JIVARO (référence ANR-08-BLAN-0259-02). Je remercie Jean-Loup Waldspurger pour les suggestions et commentaires utiles qu'il m'a apporté.

\section{Introduction} 
\label{sec:introduction}

Soit $F$ un corps local non archimédien de caractéristique nulle, c'est-dire une extension finie de $\Q_p$ pour un certain nombre premier $p$. On note $\lie{o}$ l'anneau des entiers de $F$, $p$ la caractéristique résiduelle de $F$, $\F_q$ le corps résiduel de $F$ (ici $q$ est le cardinal du corps résiduel). On se donne $\varpi \in \lie{o}$ une uniformisante, on note $|\cdot|_F$ la valeur absolue de $F$ normalisée de telle sorte que $|\varpi| = q^{-1}$, et $v_F$ la valuation normalisée par $v_F(\varpi) = 1$.

Par convention, si $\alg{X}$ est une variété algébrique définie sur $F$ (noté par une lettre grasse), on note $X = \alg{X}(F)$ (même lettre non grasse) l'ensemble de ses points rationnels sur $F$. En particulier on se donne $\alg{G}$ un groupe algébrique défini sur $F$, réductif, connexe et quasi-déployé, $\alg{\g}$ son algèbre de Lie, et on note $G = \alg{G}(F)$ et $\g = \alg{\g}(F)$ les ensembles respectifs de leurs points sur $F$.

Si $X$ est un espace topologique localement compact totalement discontinu, on note $\hecke(X)$ l'espace des fonctions à valeurs complexes sur $X$ localement constantes à support compact, et son dual $\distr(X) = \hecke(X)^*$ l'espace des distributions sur $X$.

Si $G$ agit sur $X$, alors on en déduit des actions sur les espaces $\hecke(X)$ et $\distr(X)$ (via $(g.f)(x) = f(g^{-1}.x)$ et $D(f) = D(g^{-1}.f)$ pour $g \in G$, $x \in X$, $f \in \hecke(X)$ et $D \in \distr(X)$). Si $x \in X$, alors on note $G.x$ l'orbite de $x$, et l'intégrale orbitale en $x$, est la distribution $J_x$ qui à $f \in \hecke(X)$ associe l'intégrale de $f$ sur l'orbite $G.x$ de $x$ (sous réserve de convergence et avec le choix d'une mesure convenable sur $G.x$).

\[
	J_x(f) = \int_{G.x} f(y) \, \D y = \int_{G/G_x} f(g.x) \, \D g
\]

On considère $\alg{G} = \GL_{N} \rtimes \scal{\theta}$ où $\theta$ agit sur $\GL_N$ par l'automorphisme involutif $g \mapsto g^{-t}$ (où $g^{-t}$ désigne l'inverse de la transposée
C'est un groupe réductif non connexe dont on note $\alg{G^0}$ la composante neutre, et $\widetilde{\alg{G^0}}$ l'autre composante irréductible. On note $G = \alg{G}(F)$, $G^0 = \alg{G^0}(F)$ et $\widetilde{G^0} = \widetilde{\alg{G^0}}(F)$.

L'ensemble des $\gamma \theta \in \widetilde{G^0}$ tels que $\gamma$ est antisymétrique forme une classe de conjugaison stable. On considère l'intégrale orbitale stable associée, c'est-à-dire la distribution $\OrbInt$ qui à une fonction $f \in \hecke(G)$ associe l'intégrale de $f$ sur cette classe de conjugaison. Chenevier et Clozel proposent dans \cite{Chenevier} une expression conjecturale de cette intégrale orbitale comme une intégrale sur le spectre tempéré auto-dual de $G^0$ de la forme
\[
	\OrbInt(f) = \int_{\Irr_{\temp}^{\theta}(G^0)} \Tr_{\widetilde{G^0}}(\pi^+(f)) \, \D \pi
\]
où $\Irr_{\temp}^{\theta}(G^0)$ désigne l'ensemble des (classes d'isomorphisme de) représentations de $G^0$ qui sont tempérées irréductibles et $\theta$-stables (ou, ce qui revient au même, auto-duales), $\pi^+$ est un prolongement de $\pi$ à $G$ obtenu via le choix d'un entrelacement $\pi(\theta)$ entre $\pi$ et $\pi \circ \theta$ involutif\,\footnote{Si on réalise les représentations $\pi$ et $\pi \circ \theta$ dans le même espace vectoriel.}, $\Tr_{\widetilde{G^0}}\pi^+$ désigne la restriction à $\widetilde{G^0}$ du caractère de $\pi^+$, et $\D \pi$ serait une mesure positive à support dans les représentations dites symplectiques, c'est-à-dire dont le paramètre de Langlands préserve une forme symplectique.

Nous établissons une telle expression dans le théorème \ref{thm:conjecture}. Notre démarche est la suivante. Le groupe $\alg{G'} = \SO(2n+1)$ est un groupe endoscopique de $\alg{G}$ et les distributions stables sur $G'$ se transfèrent vers $\widetilde{G^0}$ par transfert endoscopique. Notamment on montre dans la proposition \ref{prop:transfertIntegraleOrbitale} que le transfert de l'intégrale orbitale en $1$ sur $G'$ est $\lambda \OrbInt_{\widetilde{G^0}}(\eta, \cdot)$ avec $\lambda > 0$. Ce qui permet, en appliquant la formule de Plancherel sur $G'$ d'obtenir une expression du type
\[
	\OrbInt_{\widetilde{G^0}}(\eta, f) = \frac{1}{\lambda} \int_{\Irr_{\temp}(G')} \Tr(\pi'(f')) \D \pi'
\]
Où $\Irr_{\temp}(G')$ désigne l'ensemble des (classes d'isomorphisme de) représentations de $G'$ qui sont tempérées irréductibles, $f' : G' \to \C$ est un transfert endoscopique de $f$ et $\D \pi'$ est la mesure de Plancherel sur $G'$. Pour terminer, on transfère l'intégrande à $\widetilde{G^0}$. En effet, la mesure de Plancherel est constante sur les $L$-paquets de $G'$ (voir corollaire \ref{cor:plancherelLpaquets}), on peut donc réunir les termes d'un même $L$-paquet. Puis en utilisant des résultats de J. Arthur (conditionnels à la stabilisation de la formule des traces tordue), si $\Pi'$ est un $L$-paquet tempéré de $G'$ et $\sum_{\pi' \in \Pi'} \Tr \pi'$ son caractère, alors c'est une distribution stable qui se transfère à $\widetilde{G^0}$ en la distribution $\Tr_{\widetilde{G^0}} \pi^+$ pour une unique représentation $\pi$ tempérée symplectique de $G^0$ (voir \ref{thm:arthur}).


\section{Le groupe $\GL_{2n}$ tordu} 
\label{sec:groupe_symplectique}


\subsection{Les automorphismes $\theta$ et $\theta_0$} 
\label{sub:l_automorphisme_theta_}

Dans ce document, on s'intéresse au ``groupe $\GL_N$ tordu'' pour $N$ pair. Décrivons cela plus en détails. On fixe un entier $n \in \N^*$, et on pose $N = 2n$ et $\alg{G}^0 = \GL_{N}$ et $G^0 = \alg{G}^0(F)$ le groupe des points sur $F$. On définit l'automorphisme $\theta : \alg{G}^0 \to \alg{G}^0$ par $\theta(g) = g^{-t}$, où $g^{-t}$ désigne l'inverse de la transposée de $g$ (ou encore la transposée de l'inverse puisque les deux opérations commutent). Par ailleurs on définit la matrice $J_0 \in G^0$ par
\[
	J_0 = \pars{
	\begin{matrix}
		&&&&1\\
		&&&-1&\\
		&&\revdots&&\\
		&1&&&\\
		-1&&&&\\
	\end{matrix}
	}
\]
La matrice $J_0$ vérifie $J_0^2 = - I_{N}$ et $\theta(J_0) = J_0$. On définit l'automorphisme $\theta_0 : \alg{G}^0 \to \alg{G}^0$ pour tout $g \in \alg{G}^0$ par
\[
	\theta_0(g) = J_0 \, g^{-t} J_0^{-1} = \theta(J_0 g J_0^{-1})
\]
C'est-à-dire que l'on a $\theta_0 = \theta^{J_0}$. Les automorphismes $\theta$ et $\theta_0$ sont d'ordre $2$ (ie ce sont des involutions). On considère $\alg{G} = \alg{G}^0 \rtimes \langle \theta \rangle$ le groupe $\alg{G}^0$ tordu par $\theta$
, et on note $G = \alg{G}(F)$ le groupe de ses points sur $F$. C'est un groupe réductif non connexe dont la composante neutre est $\alg{G}^0$ et le groupe des composantes $\alg{G} / \alg{G}^0 \simeq \langle \theta \rangle \simeq \Z / 2 \Z$.

L'autre composante irréductible est $\widetilde{\alg{G}^0} = \alg{G}^0 \theta$, et on note de même $\widetilde{G^0} = \widetilde{\alg{G}^0}(F)$. Le couple $(\alg{G}^0, \widetilde{\alg{G}^0})$, c'est un espace tordu au sens de  Labesse (voir \cite{Labesse} chapitre 2, sections 1 et 2). En particulier, $\widetilde{\alg{G}^0}$ est une variété algébrique isomorphe à $\alg{G}^0$ via $f_{\theta} : g_0 \mapsto g_0 \theta$, sur laquelle le groupe $\alg{G}^0$ agit transitivement à gauche et à droite par multiplication.

On a la décomposition $\alg{G} = \alg{G}^0 \sqcup \widetilde{\alg{G}^0}$. On remarque que l'on peut remplacer $\theta$ par n'importe quel conjugué (notamment $\theta_0$) dans la définition de $\alg{G}$ et $\widetilde{\alg{G}^0}$ (puisque $\theta$ et $\theta_0$ sont congrus modulo $\Inn(G)$, on a $\alg{G}^0 \rtimes \langle \theta_0 \rangle = \alg{G}^0 \rtimes \langle \theta \rangle$.
Au final, le choix de $\theta$, $\theta_0$ (ou d'un autre conjugué) porte à peu de conséquences à quelques ajustements mineurs près qu'on précise dans la suite.

\subsection{Une représentation linéaire de $G$} 
\label{sub:representation_lineaire_de_g_}

Le groupe $\alg{G}$ est linéaire algébrique donc peut se plonger dans un groupe linéaire $\GL_k$ pour un certain $k \in \N$. Notamment, on dispose d'un plongement $i : \alg{G} \to \GL_{2N}$ explicite donné par

\[
	i(g_0) = \mat{g_0 & 0\\ 0& \theta(g_0)} \qquad \textrm{ et } \qquad i(\theta) = \mat{0 & I_N\\ I_N& 0}
\]
pour $g_0 \in \alg{G}^0$. Alors on a

\[
	i(\widetilde{\alg{G}^0}) = \set{\mat{0 & g_0 \\ \theta(g_0) & 0}, \, g_0 \in \alg{G}^0}
\]
Et d'après la définition, on a par ailleurs
\[
	i(\theta_0) = \mat{0 & J_0\\ J_0 & 0}
\]

\begin{lemme}
	Soit $g_0 \in G^0$, alors $g_0 \theta \in G$ est semi-simple (resp. semi-simple fortement régulier) si et seulement si $g_0\theta(g_0) \in G^0$ l'est.
\end{lemme}
\begin{proof}
	D'abord on remarque que $g_0 \in G^0$ est semi-simple (respectivement nilpotent) si et seulement si $\theta(g_0)$ l'est. En conséquence, $g_0 \in G^0$ est semi-simple si et seulement si $i(g_0)$ l'est.
	Comme $(g_0\theta)^2 = g_0 \theta(g_0)$, il est clair que si $g_0\theta$ est semi-simple, alors $g_0 \theta(g_0)$ l'est aussi. Réciproquement, supposons $g_0 \theta(g_0)$ semi-simple. On écrit la décomposition de Jordan-Chevalley de $g_0 \theta = g_{ss} g_u$ avec $g_{ss}, g_u \in G$ qui commutent, $i(g_{ss})$ semi-simple et $i(g_u)$ unipotent. On a alors $g_0 \theta(g_0) = g_{ss}^2 g_u^2$. Mais par unicité de la décomposition de Jordan-Chevalley, $g_u^2 = 1$, donc $g_u = 1$.	
	
	 Par ailleurs, les polynômes caractéristiques de $g_0\theta(g_0)$ et $i(g_0\theta)$ sont reliés par la relation  $\chi_{i(g_0 \theta)}(X) = \chi_{g_0 \theta(g_0)}(X^2)$ (il s'agit d'un calcul de déterminant par blocs). En particulier, $\chi_{g_0 \theta(g_0)}$ est scindé à racines simples si et seulement si $\chi_{i(g_0 \theta)}$ l'est aussi ($0$ n'est pas racine). Ce qui conclut pour le cas fortement régulier.
\end{proof}


Soit $U \subset G^0$ le sous-groupe de $G^0$ des matrices triangulaires supérieures unipotentes. Soit $\phi : F \to \C^*$ un caractère de $F$, alors on définit un caractère $\varphi : U \to \C^*$ de $U$ par la formule

\[
	\varphi \mat{
	1 & x_1		& *\\
	  & \ddots	& x_{N-1}\\
	  &			& 1 \quad
	} = \phi(x_1 + \cdots + x_{N-1})
\]
On vérifie alors que $\theta_0(U) = U$, et que $\varphi = \varphi \circ \theta_0$.

\subsection{Représentations lisses de $G^0 \rtimes \langle \theta \rangle$ et représentations $\theta$-stables de $G^0$} 
\label{sub:representations_de_g_rtimes_}

Les groupes $G$ et $G^0$ sont des groupes localement compacts totalement discontinus. On note $\Rep(G)$ et $\Rep(G^0)$ les catégories des représentations complexes lisses de $G$ et $G^0$ respectivement. De même on note $\Hecke(G)$ et $\Hecke(G^0)$ les algèbres de Hecke correspondantes, et $\Hecke(\widetilde{G^0})$ le sous-espace des fonctions à support dans $\widetilde{G^0}$. On dispose d'une injection naturelle $\Hecke(G^0) \hookrightarrow \Hecke(G)$, qui munit $\Hecke(G)$ d'une structure de $\Hecke(G^0)$-module à gauche et à droite et pour laquelle $\Hecke(\widetilde{G^0})$ est un sous-module. En tant que $\Hecke(G^0)$-bi-module, on a la décomposition

\[
	\Hecke(G) = \Hecke(G^0) \oplus \Hecke(\widetilde{G^0}) 
\]
L'application $f \mapsto f * \delta_{\theta}$, réalise une bijection entre $\Hecke(\widetilde{G^0})$ et $\Hecke(G^0)$ qui rend $\Hecke(\widetilde{G^0})$ est isomorphe à $\Hecke(G^0)$ en tant que module à gauche, et l'action à droite d'une fonction $f$ revient sur $\Hecke(G^0)$ à la multiplication par (produit de convolution avec) $f \circ \theta$.

	On dira qu'une représentation $(\pi, V) \in \Rep(G^0)$ de $G^0$ est $\theta$-stable s'il existe un $G$-isomorphisme entre $(\pi,V)$ et $(\pi \circ \theta, V)$. Et on note $\Rep(G^0)^{\theta}$ la sous-catégorie pleine de $\Rep(G^0)$ des représentations $\theta$-stables.
	

%
%
%

Si $(\pi, V) \in \Rep(G)$ est une représentation lisse de $G$, on note $\pi_0 = \pi_{|G^0}$ et $\widetilde{\pi_0} = \pi_{|\widetilde{G^0}}$\,\footnote{il ne s'agit par d'une représentation puisque $\widetilde{G^0}$ n'est pas un groupe} alors $(\widetilde{\pi_0}, \pi_0)$ est une représentation tordue pour le couple $(\widetilde{G^0}, 1)$ au sens de \cite{Labesse} (chapitre 2, section 3). Une représentation lisse $(\pi,V)$ de $G$ est entièrement déterminée par la donnée du triplet $(V, \pi_{|G^0}, \pi(\theta))$. On vérifie que $\pi(\theta)$ est un automorphisme d'ordre $2$ de $V$ qui entrelace $\pi$ et $\pi \circ \theta$, et la représentation $\pi_{|G^0}$ est donc $\theta$-stable. Réciproquement, étant donné un triplet $(V, \pi, A)$ où $(\pi,V)$ est une représentation lisse $\theta$-stable de $G^0$ et $A$ un automorphisme d'ordre $2$ de $V$ qui entrelace $\pi$ et $\pi \circ \theta$, alors on peut construire une représentation lisse de $G$ en posant $\pi(\theta) = A$. Une représentation lisse de $G$ est donc la donnée d'une représentation $\theta$-stable de $G^0$ et d'un choix d'un isomorphisme $A \in \Hom_{G^0}(\pi, \pi \circ \theta)$ d'ordre $2$ en tant qu'automorphisme de $V$. Notons qu'à $G$-isomorphisme près, l'opérateur $A$ est unique au signe près. En effet, si $\pi, \pi' \in \Rep(G)$ sont des représentations de $G$, leurs restrictions à $G^0$ sont isomorphes si et seulement si $\pi' \simeq \pi \otimes \chi$ pour $\chi$ un caractère de $\langle \theta \rangle$, ce qui veut dire que si l'on réalise $\pi$ et $\pi'$ dans le même espace vectoriel, on a $\pi'(\theta) = \pm \pi(\theta)$. En particulier, si $(\pi, V) \in \Rep(G^0)$ est une représentation lisse de $G^0$ et $\pi^+$ un prolongement à $G$, alors la restriction $\Tr_{\widetilde{G^0}}(\pi^+)$ à $\widetilde{G^0}$ du caractère de $\pi^+$ est déterminée par $\pi$ au signe près.

Dans l'autre direction, toute représentation irréductible $\theta$-stable de $G^0$ est prolongeable en une représentation de $G$. En effet, si $(\pi,V) \in \Irr(G^0)$ est une représentation irréductible $\theta$-stable de $G^0$ et $A \in \Hom_{G^0}(\pi, \pi \circ \theta)$ un isomorphisme quelconque, alors comme $A^2 \in \Hom_{G^0}(\pi,\pi)$, le lemme de Schur assure qu'il existe $\lambda \in \C$ non nul tel que $A^2 = \lambda \Id_V$. On se donne alors $\mu$ une racine carrée de $\lambda$, alors $\frac{A}{\mu}  \in \Hom_{G^0}(\pi, \pi \circ \theta)$ est d'ordre $2$ ce qui permet de prolonger $\pi$ à $G$.

Le même raisonnement s'applique évidemment en remplaçant $\theta$ par $\theta_0$, et comme les deux automorphismes sont conjugués, on vérifie qu'une représentation est $\theta$-stable si et seulement si elle est $\theta_0$-stable (en fait si $\pi \in \Rep(G^0)$, alors $\pi \circ \theta$ et $\pi \circ \theta_0$ sont toujours isomorphes via $\pi(J_0) : V \to V$). Notons enfin que pour $\pi \in \Rep(G^0)$ irréductible, on a toujours $\pi \circ \theta_0 \simeq \check{\pi}$ (ici on note $\check{\pi}$ la représentation contragrédiente, cf. \cite{Zelevinskii} theorem 7.3), donc une représentation irréductible est $\theta_0$-stable si et seulement si elle est autoduale.

\subsection{Conjugaison et conjugaison stable dans $G$} 
\label{sub:conjugaison_et_conjugaison_stable_dans_g_}

On note $\Ad : \alg{G} \to \Aut(\alg{G})$ l'action de $\alg{G}$ sur lui-même par conjugaison. Chaque composante irréductible de $\alg{G}$ est stable par cette action.

Si $g_0 \in \alg{G}^0$, alors $\widetilde{\alg{G}^0}$ est stable par $\Ad(g_0)$, et on note $\Ad_{\theta}(g_0) = f_{\theta} \circ \Ad(g_0)_{|\widetilde{\alg{G}^0}} \circ f_{\theta}^{-1}$ l'action sur $\alg{G}^0$ déduite via $f_{\theta}$, c'est-à-dire
\[
	\Ad_{\theta}(g_0)(h_0) = h_0 g_0 \theta(h_0)^{-1} = h_0 g_0 \, ^t h_0
\]
On appelle cette action, la conjugaison $\theta$-tordue. On définit la conjugaison $\theta_0$-tordue $\Ad_{\theta_0}$ de manière analogue en remplaçant $\theta$ par $\theta_0$. Une classification des classes de conjugaison semi-simples dans $\widetilde{G^0}$ est donnée dans \cite{WaldspurgerGLnTordu} I.3. 

Deux éléments de $x, y \in G = \alg{G}(F)$ semi-simples sont dits stablement conjugués s'il existe $g \in \alg{G}(\overline{F})$ tel que $x = gyg^{-1}$ et pour tout $\sigma \in \Gal(\overline{F}/F)$, on a $g^{-1} \sigma(g) \in Z(\alg{G}^0)^{\theta}\alg{G}_y^0$ (où $\Gal(\overline{F}/F)$ désigne le groupe de Galois de $\overline{F}/F$ et $\alg{G}_y^0$ la composante neutre du centralisateur de $y$ dans $\alg{G}^0$). La conjugaison dans $\alg{G}(F)$ entraine la conjugaison stable, qui entraine la conjugaison dans $\alg{G}(\overline{F})$.


\subsection{Classe de conjugaison $\theta$-tordue d'une matrice anti-symétrique} 
\label{sub:classe_de_conjugaison_tordue}

On définit $A_{2n}(F)$ comme l'ensemble de matrices carrées de taille $2n$ à coefficients dans $F$ inversibles antisymétriques.
\[
	A_{2n}(F) = \braces{\gamma \in \GL_{2n}(F), \, ^t \gamma = - \gamma, \forall i \in \ints{1,2n}, \gamma_{i i} = 0}
\]
Il est clair d'après la définition que $A_{2n}(F)$ est un fermé de $G^0$. Par ailleurs $G^0$ agit sur $A_{2n}(F)$ par $\theta$-conjugaison, c'est-à-dire
\[
	\text{Ad}_{\theta}(g).\gamma = g \gamma \, ^t g
\]
pour $\gamma \in A_{2n}(F)$ et $g \in G^0$. On définit $J_n \in A_{2n}(F)$ par
\[
	J_n = \pars{
	\begin{matrix}
		0 & I_n \\
		-I_n & 0
	\end{matrix}
	}
\]
Le groupe symplectique est défini comme le stabilisateur de $J_n$ sous cette action.

\[
	\Stab_{\text{Ad}_{\theta}}(J_n) = \Sp_{2n}(F)
\]
Et pour tous $\gamma \in A_{2n}(F)$ et $g \in G^0$,

\[
	\Stab_{\text{Ad}_{\theta}}(\text{Ad}_{\theta}(g).\gamma) = g \, \Stab_{\text{Ad}_{\theta}}(\gamma) \, g^{-1}
\]
Or cette action est transitive 
Cette propriété étant valable pour tout corps
, alors $A_n(F)$ est même la classe de conjugaison $\theta$-tordue stable de n'importe quelle matrice alternée. Par ailleurs tous les stabilisateurs sont donc conjugués entre eux, et en particulier conjugués à $\Sp_{2n}(F)$ (les ``centralisateurs tordus'' de Chenevier et Clozel sont de tels stabilisateurs). Par surcroît, pour tout $\gamma \in A_{2n}(F)$ on a donc une surjection

\begin{align*}
	.\gamma : G^0 & \longrightarrow A_{2n}(F) \\
	g & \longmapsto \text{Ad}_{\theta}(g).\gamma
\end{align*}
ce qui induit une bijection
\[
	G^0/\operatorname{Stab}_{\text{Ad}_{\theta}}(\gamma) \underset{.\gamma}{\sim} A_{2n}(F)
\]
En particulier, pour $\gamma = J_n$ on tire une bijection
\[
	\GL_{2n}(F)/\Sp_{2n}(F) \underset{.J_n}{\sim} A_{2n}(F)
\]
On choisit sur le quotient $\GL_{2n}(F)/\operatorname{Stab}_{\text{Ad}_{\theta}}(\gamma)$ une mesure invariante par translation à gauche (existe et est unique à constante près, voir par exemple \cite{Renard} II 3.9 dans le cas unimodulaire). On munit $A_{2n}(F)$ de la mesure déduite de celle sur $\GL_{2n}(F)/\operatorname{Stab}_{\text{Ad}_{\theta}}(\gamma)$ par bijection, cette mesure est donc invariante par l'action de $G^0$ par $\theta$-conjugaison, et c'est à constante près la seule qui possède cette propriété.

En résumé, on a montré dans cette partie que le centralisateur $\theta$-tordu d'une matrice de $A_{2n}(F)$ est conjugué au groupe symplectique, et que $\GL_{2n}$ quotienté par ce groupe (en particulier $\GL_{2n}/\Sp_{2n}$) est en bijection avec $A_{2n}(F)$.

Les orbites $\theta_0$-tordues se déduisent simplement des orbites $\theta$-tordues. En effet, si
$[g_0 \theta]^{G^0}$ et $[g_0 \theta_0]^{G^0}$ désignent les orbites $\theta$-tordue et $\theta_0$-tordue respectivement d'un élément $g_0 \in G^0$, on a la relation

\[
	[\gamma \theta_0]^{G^0} = [(\gamma J_0) \theta]^{G^0} . J_0^{-1}
\]
De même, le centralisateur $\theta$-tordu de $\gamma$ est le centralisateur $\theta_0$-tordu de $\gamma J_0$. Et les intégrales orbitales $\theta_0$-tordues se déduisent des intégrales orbitales $\theta$-tordues via
\[
	\OrbInt_{G^0 \theta_0}(\gamma,f) = \OrbInt_{G^0 \theta}(\gamma J_0,\lambda(J_0).f)
\]
Où $\OrbInt_{G^0 \theta_0}(\gamma,f)$ désigne l'intégrale orbitale de $f$ sur la classe de conjugaison $\theta_0$-tordue de $\gamma$ (même chose avec $\theta$), et $\lambda(J_0).f$ est la fonction $\lambda(J_0).f : g \mapsto f(g J_0^{-1})$. Moyennant ces adaptations mineures, on peut donc passer à loisir de $\theta$ à $\theta_0$ selon ce qui nous arrange (en particulier on peut reprendre les résultat de \cite{WaldspurgerGLnTordu}, \cite{Shahidi}, qui utilisent $\theta$, ou ceux de \cite{Chenevier}, qui utilisent $\theta_0$).


\section{Le problème de Chenevier et Clozel} 
\label{sec:conjecture_de_chenevier_clozel}

\subsection{Le problème} 
\label{sub:la_conjecture}

On se place toujours dans le même cadre, on pose
\[
	J_0 = \pars{
	\begin{matrix}
		&&&&1\\
		&&&-1&\\
		&&\revdots&&\\
		&1&&&\\
		-1&&&&\\
	\end{matrix}
	} = \pars{
	\begin{matrix}
		0&K_0\\
		-\,^t K_0&0
	\end{matrix}
	} \quad \textnormal{ et } \quad \gamma_0 = \pars{
	\begin{matrix}
		I_n&0\\
		0&-I_n
	\end{matrix}
	}
\]
On définit $\delta_0 = \gamma_0 \, \diag(1,-1, \ldots, -1)$. On vérifie que le produit $\delta_0 J_0$ est dans $A_{2n}(F)$
\[
	\delta_0 J_0 = \pars{
	\begin{matrix}
		&&&&&-1\\
		&&&&\revdots&\\
		&&&-1&&\\
		&&1&&&\\
		&\revdots&&&&\\
		1&&&&&\\
	\end{matrix}}
\]
L'intégrale orbitale $\theta_0$-tordue considérée par Chenevier et Clozel est pour $f \in \mathcal{C}_c^{\infty}(G^0)$
\[
	\text{TO}_{\delta_0}(f) = \OrbInt_{G^0}(\delta_0 \theta_0,f*\delta_{\theta_0}) = \int_{G^0/I_{\delta_0}} f(g \, \delta_0 \, \theta_0(g^{-1})) \, \D g = \OrbInt_{G^0}(\delta_0 J_0 \theta,\lambda(J_0).f*\delta_{\theta_0})
\]
où $I_{\delta_0}$ est le stabilisateur $\theta_0$-tordu de $\delta_0$, c'est-à-dire
\[
	I_{\delta_0} = \braces{g \in G^0, \, g \, \delta_0 \, \theta_0(g^{-1}) = \delta_0}^0
\]
on notera $\Irr(G^0)^{\theta_0}_{\textrm{temp}}$ l'ensemble des (classes d'isomorphisme de) représentations irréductibles tempérées et auto-duale de $G^0$. On dit qu'une représentation irréductible auto-duale de $G^0$ est symplectique si son paramètre de Langlands 	
est symplectique (c'est-à-dire préserve une forme bilinéaire alternée non dégénérée). Une représentation irréductible tempérée symplectique s'écrit sous la forme

\[
	\pi = \Ind_P^G(\alpha_1 \otimes \alpha_1^{\theta} \otimes \ldots \otimes \alpha_r \otimes \alpha_r^{\theta} \otimes \beta_{1} \otimes \ldots \otimes \beta_s)
\] 
Où les $\alpha_i$ et $\beta_j$ sont essentiellement de la série discrète et les représentations $\beta_j$ sont auto-duales symplectiques. Les éléments de l'orbite de $\pi$ qui sont encore symplectiques sont ceux obtenus en tordant les $\alpha_i$ par un caractère $\chi_i$ non-ramifié (et le $\alpha_i^{\theta}$ par l'inverse de $\chi_i$) et les $\beta_i$ par des caractères d'ordre $2$. En fait, les éléments de l'orbite de $\pi$ sont symplectiques dès lors qu'ils sont auto-duaux.

\begin{conj}[Chenevier, Clozel]
	Il existe une mesure positive $\D \pi$ sur le spectre tempéré auto-dual de $G^0$, à support dans les représentations tempérées symplectiques telle que l'on ait pour tout $f \in \Hecke(G^0)$,
	\[
		\OrbInt_{G^0}(\delta_0 \theta_0,f) = \int_{G^0/I_{\delta_0}} f(g \delta_0 \theta_0(g^{-1})) \, \D g = \int_{\Irr(G^0)^{\theta_0}_{\textrm{temp}}} \Tr[\pi(\theta_0) \, \pi(f)] \, \D \pi
	\]
	où  $\pi(\theta_0) \in \Hom_{G^0}(\pi, \pi \circ \theta_0)$ est la normalisation de Whittaker (cf. \cite{Chenevier} 4.1, et page 43).
\end{conj}
On peut reformuler en terme de $\theta$, si on pose $\delta = \delta J_0$, alors la conjecture équivaut à supposer l'existence de $\D \pi$ telle que
\[
	\int_{G^0/I_{\delta_0}} f(g \delta g^t) \, \D g = \int_{\Irr(G^0)^{\theta_0}_{\textrm{temp}}} \Tr[\pi(\theta) \, \pi(f)] \, \D \pi
\]
On peut également, remplacer $\delta_0$ par n'importe quelle matrice $\delta$ telle que $\delta J_0$ soit antisymétrique inversible. La conjecture concernant la positivité de la mesure vient probablement d'une analogie avec \cite{Chenevier} proposition 4.15. Grossièrement, on peut expliquer rapidement la conjecture par le calcul formel suivant : on peut appliquer la formule de Plancherel à l'intégrande dans $\OrbInt_{G^0\theta_0}(\delta_0,f)$, et après interversion des intégrales (la convergence n'est en fait pas garantie), on trouverait une intégrale sur $\Irr(G^0)^{\theta_0}_{\textrm{temp}}$ (on montre facilement que l'intégrande est nul si la représentation n'est pas auto-duale) avec pour intégrande les intégrales orbitales de coefficients matriciels (modulo le centre) considérées dans \cite{Shahidi} proposition 5.3. On sait d'après \cite{Henniart2} et \cite{Shahidi} prop 5.1, 5.3 que dans le cas ou la représentation est cuspidale, cette intégrale orbitale est non nulle si et seulement si la représentation est symplectique. Malheureusement, pour une représentation tempérée plus générale, on ne sait même pas si l'expression est convergente a priori. En substance c'est donc l'objet de cette conjecture. Nous donnons ce calcul en détail dans la section suivante à titre indicatif.

\section{Passage à l'algèbre de Lie} 
\label{sec:passage_a_l_algebre_de_lie}

On reproduit ici des définitions et résultats de \cite{HarishDeBacker} et \cite{WaldspurgerEndoscopieTordue} dont on se sert dans la suite.

\subsection{Conjugaison et orbites dans l'algèbre de Lie} 
\label{sub:conjugaison_et_orbites}
On suppose dans cette section $G$ connexe (sauf mention explicite du contraire). Soit $\g$ l'algèbre de Lie de $G$, c'est un $F$-espace vectoriel de dimension finie sur lequel $G$ agit par conjugaison. Si $x \in G$ et $X \in X$, on notera $x.X = \Ad(x)X$ cette action.
\begin{ex}
	Dans le cas où $G = \GL_n(F)$, alors $\g = \gl_n(F) = \mathcal{M}_n(F)$ l'algèbre des matrices carrés de taille $n$ munie du crochet de Lie $[X, Y] = XY - YX$, et l'action de $x \in G$ est donnée par $x.X = x X x^{-1}$. Dans d'un groupe réductif général, il existe des morphismes injectifs $G \hookrightarrow \GL_n(F)$, $\g \hookrightarrow \gl_n(F)$ qui préservent les structures afférentes.
	
	Dans le cas où $G = \GL_n(F) \rtimes \scal{\theta}$, alors on a aussi $\g = \gl_n(F)$ puisque l'algèbre de Lie ne dépend que de la composante neutre (alternativement, si $\theta$ est d'ordre $k$, on peut plonger $G$ dans $\GL_{nk}(F)$, et l'image du morphisme $\g \to \gl_n(F)$ qui s'en déduit s'identifie à $\gl_n(F)$). Si $\theta : \GL_n(F) \to \GL_n(F)$ est donné par $g \mapsto g^{-t}$, alors la conjugaison par $\theta$ est donnée pour $X \in \gl_n(F)$ par 
	\[
		\theta.X = -X^t
	\]
\end{ex}

De l'action de $G$ sur $\g$, on tire des actions sur l'espace $\hecke(\g)$ des fonctions localement constantes à support compact sur $\g$, et celui $\distr(\g) = \hecke(\g)^*$ des distributions sur $\g$. En effet, si $f \in \hecke(\g)$, et $x \in G$, on définit $f^x \in \hecke(\g)$ pour $X \in \g$ par
\[
	f^x(X) = f(x.X)
\]
Et si $T \in \distr(\g)$ est une distribution sur $\g$, on définit la distribution $^xT$ pour $f \in \hecke(\g)$ par
\[
	^xT(f) = T(f^x)
\]
On dit qu'une distribution $T \in \distr(\g)$ est invariante par $G$-conjugaison (ou $G$-invariante) si $^x T = T$ pour tout $x \in G$. On note $\distr(\g)^G$ l'espace vectoriel des distributions $G$-invariantes.

On appelle un $G$-domaine dans $\g$ un sous-ensemble de $\g$ invariant par $G$-conjugaison qui est à la fois ouvert et fermé. Pour deux sous-ensembles $S \subset G$ et $\omega \subset \g$, on pose
\[
	\omega^S = \bigcup_{s \in S} \textrm{Ad}(s) \omega
\]
Et on note $\distr(\g)^G(\omega)$ le sous-espace de $\distr(\g)^G$ des distributions $T \in \distr(\g)^G$ telles que $\Supp T \subset \overline{\omega^G}$ (la barre désigne la fermeture pour la topologie $p$-adique).
Si $X \in \g$, on notera $X^S$ pour $\set{X}^S$. 

Par une orbite (ou plus précisément une $G$-orbite) dans $\g$, on entend un ensemble de la forme $X^G$ pour $X \in \g$. Si $X \in \g$, on pose $C_{\g}(X) = \set{Y \in \g, \, [X,Y] = 0} = \ker \Ad(X)$ le centralisateur dans $\g$ du point $X$. C'est un sous-espace vectoriel de $\g$ dont la dimension ne dépend que de l'orbite $X^G$ de $X$ car si $g \in G$, alors $C_{\g}(\Ad(g) X) = \Ad(g^{-1}) \, C_{\g}(X)$. Cela autorise la définition suivante.

\begin{defn}[Dimension et rang d'une orbite]
	Soit $\Orb$ est une $G$-orbite dans $\g$, et $X \in \Orb$ un point quelconque de l'orbite. On définit

	\[
		d(\Orb) = \dim(\g/C_{\g}(X)) \qquad \textrm{ et } \qquad r(\Orb) = \dim C_{\g}(X)
	\]
	On dit que $d(\Orb)$ est la dimension de l'orbite, et $r(\Orb)$ est le rang de l'orbite. On note $\mathcal{N} \subset \g$ l'ensemble des éléments nilpotents de $\g$.
\end{defn}


\subsection{Conjugaison et orbites stables} 
\label{sub:conjugaison_et_orbites_stables}

Deux éléments $X,Y \in \g$ semi-simple réguliers sont dits stablement conjugués s'il existe $g \in G(\overline{F})$ tel que $Y = g.X$. Si $\omega \subset \g$, on note $\omega^{G, \st}$ l'ensemble des éléments stablement conjugués à un élément de $\omega$ (de manière générale, on notera par un exposant ou un indice $\st$ les objets afférents à la conjugaison stable). Si $X \in \g$, on notera $X^{G, \st}$ pour $\set{X}^{G, \st}$. Une orbite stable est une classe d'équivalence pour cette relation, c'est à dire un ensemble de la forme $X^{G, \st}$ pour $X \in \g$.

Quand $G = \GL_n(F)$, alors les notions de conjugaison stable et ordinaire coïncident. Mais ce n'est en général pas le cas pour d'autres groupes.
Comme la conjugaison ordinaire entraine la conjugaison stable, une orbite stable est une union d'orbites ordinaires. On peut montrer que cette union est finie.


\subsection{Voisinage d'un élément semi-simple} 
\label{sub:voisinage_d_un_element_semi_simple}

Soit $\gamma \in \g$ est un élément semi-simple de $\g$. On note $M = C_G(\gamma) = \set{g \in G, \, g.\gamma = \gamma}$ le centralisateur de $\gamma$ dans $G$, et $\lie{m} = C_{\g}(\gamma) = \set{Y \in \g, \, [\gamma,Y] = 0} = \ker \Ad(\gamma)$ son algèbre de Lie.

\begin{defn}
	On note $\Orb_{\g}(\gamma)$ l'ensemble des $G$-orbites $\Orb$ de $\g$ telles que $\gamma \in \overline{\Orb}$. Pour $d \in \N$, on note $\Orb_{\g}(\gamma)_d$ le sous-ensemble des orbites de degré $d$.
\end{defn}

\begin{lemme}[\cite{HarishDeBacker}]
	\label{lemme:voisinage_semisimple}
	L'ensemble $\Orb_{\g}(\gamma)$ est fini et égal à l'ensemble des orbites de la forme $\Orb = (\gamma + Y)^G$ où $Y$ est un élément nilpotent de $\lie{m}$.
	
	Plus précisément, si on note $\Orb_{\lie{m}}(0)$ l'ensemble des $M$-orbites nilpotentes de $\lie{m}$, alors l'application
	\begin{align*}
		\Orb_{\lie{m}}(0) & \longrightarrow \Orb_{\g}(\gamma) \\
		\xi \quad & \longmapsto (\gamma + \xi)^G
	\end{align*}
	est une bijection. Il existe $U$ un $M$-domaine tel que si $\Orb = (\gamma + \xi)^G$ est dans $\Orb_{\g}(\gamma)$, alors $\Orb \cap U = \gamma + \xi$. Ce qui fournit la réciproque. Et on a l'égalité
	\[
		r(\Orb) = r(\xi)
	\]
\end{lemme}
\begin{proof}
	Il s'agit d'une concaténation des lemmes 4.7, 4.8, 4.9 et corollaires 4.10 et 4.11 de \cite{HarishDeBacker}. Dans le cas particulier où $G = \GL_n(F)$, la preuve est élémentaire; donnons-la en détail.
	
	Puisque $\Orb_{\g}(\gamma)$ ne dépend que de la classe de conjugaison de $\gamma$, on peut supposer, quitte à conjuguer, que $\gamma$ est diagonale de la forme
	\[
		\gamma = \mat{\lambda_1 I_{n_1} & &\\ & \ddots & \\ &&\lambda_s I_{n_s}}
	\]
	avec les $\lambda_i$ deux à deux distincts. Dans ce cas, $M \simeq \GL_{n_1}(F) \times \ldots \times \GL_{n_s}(F)$ ($M$ est le sous-groupe de $\GL_n(F)$ des matrices diagonales par blocs de taille $n_1, \ldots, n_s$). De même on a $\lie{m} \simeq \gl_{n_1}(F) \times \ldots \times \gl_{n_s}(F)$.
	
	Soient $N, N' \in \lie{m}$, alors on vérifie que $(\gamma + N)^G = (\gamma + N')^G$ si et seulement si $N^M = (N')^M$. En effet, si $N^M = (N')^M$, alors $(\gamma + N)^M = (\gamma + N')^M$ (puisque $M = C_G(\gamma)$), et donc $(\gamma + N)^G = (\gamma + N')^G$. Réciproquement, on se donne $g \in G$ tel que $g.\gamma + g.N = \gamma + N'$. Par unicité de la décomposition de Jordan-Chevalley, on tire $g.\gamma = \gamma$ et $g.N = N'$. Or $g.\gamma = \gamma$  signifie précisément que $g \in M$, donc au final $N^M = (N')^M$. L'équivalence assure donc que la fonction $f : \xi \quad \longmapsto (\gamma + \xi)^G$, de $\Orb_{\lie{m}}(0)$ vers l'ensemble des $G$-orbites de $\g$ est bien définie et injective. Il reste encore à vérifier que l'on a $f(\Orb_{\lie{m}}(0)) = \Orb_{\g}(\gamma)$.
	
 	Commençons par l'inclusion $f(\Orb_{\lie{m}}(0)) \subset \Orb_{\g}(\gamma)$. Soit $N \in \lie{m}$, on veut montrer que $\Orb = (\gamma + N)^G$ est dans $\Orb_{\g}(\gamma)$. Quitte à conjuguer par un élément de $M$, on peut supposer que $N = (m_{i,j})$ est triangulaire supérieure, c'est-à-dire $m_{i,j} = 0$ si $j \le i$. Si $d = \diag(x, x^2, \ldots, x^n)$, alors $d \in M$, donc $\Orb = (\gamma + d.N)^G$. Or $d.N = (x^{i-j} m_{i,j})$, donc en choisissant $x$ suffisamment grand, on peut rendre $d.N$ arbitrairement proche de $0$, ce qui prouve que $\Orb \in \Orb_{\g}(\gamma)$.
	
	Pour l'autre inclusion, soit $\Orb \in \Orb_{\g}(\gamma)$. Soit $X \in \Orb$ un point de l'orbite, on écrit sa décomposition de Jordan-Chevalley $X = D + N$ avec $D$ semi-simple et $N$ nilpotent qui commutent. Pour tout $Y \in \Orb$, on a égalité des polynômes caractéristiques $\chi_Y = \chi_X = \chi_D$, et par continuité de la fonction $Y \mapsto \chi_Y$, c'est encore vrai pour tout $Y \in \overline{\Orb}$. En particulier, on a $\chi_{\gamma} = \chi_D$. Quitte à changer $X$ en un conjugué, on peut donc supposer $D = \gamma$, et l'hypothèse que $D$ et $N$ commutent entraine alors que $N \in \lie{m}$.
	
	Pour montrer la finitude de $\Orb_{\g}(\gamma)$, montrons celle de $\Orb_{\lie{m}}(0)$. Avec les notations évidentes, on vérifie que $\Orb_{\lie{m}}(0) = \Orb_{\gl_{n_1}}(0) \times \ldots \times \Orb_{\gl_{n_s}}(0)$. On est donc ramené à montrer que $\Orb_{\gl_n}(0)$ est fini. Or deux matrices nilpotentes de $\gl_n(F)$ sont conjuguées (dans $\GL_n(\overline{F})$ et donc dans $\GL_n(F)$) si et seulement si elles ont à l'ordre près des blocs de la même forme normale de Jordan. Les éléments de $\Orb_{\gl_n}(0)$ sont donc paramétrés par les partitions de l'entier $n$, et en particulier c'est un ensemble fini.
	
	Enfin, pour la dernière partie de la preuve (l'existence du $M$-domaine $U$ et l'égalité des dimensions qui en découle), on renvoie à \cite{HarishDeBacker} lemme 4.10.
\end{proof}

Si $\gamma$ est un élément semi-simple de $G$, on peut formuler un lemme analogue décrivant l'ensemble $\Orb_G(\gamma)$ des classes de conjugaison $\Orb$ dans $G$ tels que $\gamma \in \overline{\Orb}$, ces orbites sont de la forme $(\gamma U)^G$ avec $U$ une classe de $M$-conjugaison d'éléments unipotents où $M = C_G(\gamma)$ (la preuve est analogue dans le cas $G = \GL_n$).


\subsection{Intégrales orbitales} 
\label{sub:integrales_orbitales_lie}

Soit $\Orb$ une $G$-orbite, et $X \in \Orb$ un point de l'orbite. Le centralisateur $C_G(X)$ de $X$ dans $G$ est unimodulaire, donc on peut munir $C/C_G(X)$ d'une mesure positive invariante $\D x^*$ unique à proportionnalité près\footnote{On va choisir une normalisation dans la suite.}.
On définit une distribution $\OrbInt_{\g}(X, .) = \OrbInt_{\g}(\Orb, .)$, appelée intégrale orbitale en $\Orb$ (ou en $X$), pour $f \in \hecke(\g)$ par

\[
	\OrbInt_{\g}(\Orb, f) = \int_{G/C_G(X)} f(x.X) \, \D x^*
\]
C'est bien défini car l'intégrale ci-dessus est convergente d'après \cite{Rao}. Il est clair que la distribution $\OrbInt_{\g}(\Orb, .)$ est invariante par conjugaison, c'est-à-dire $\OrbInt_{\g}(\Orb, .) \in \distr(\g)^G$.


\subsection{Homogénéité des intégrales orbitales nilpotentes} 
\label{sub:homogeneite_des_integrales_orbitales}

Pour $t \in F^*$ et $f \in \hecke(\g)$, on pose $f_t(X) = f(t X)$. Si $T  \in \distr(\g)$ est une distribution, on définit la distribution $\rho(t)T$ pour $f \in \hecke(\g)$ par
\[
	[\rho(t)T](f) = T(f_t)
\]
On a
\[
	[\rho(t) \OrbInt_{\g}(\Orb, .)](f) = \int_{G/C_G(X)} f(x.(tX)) \, \D x^*
\]
Et par unicité de la mesure invariante sur $C/C_G(tX) = C/C_G(X)$, il existe une constante $c_{\Orb}(t) > 0$ telle que
\[
	\rho(t) \OrbInt_{\g}(\Orb, .) = c_{\Orb}(t) \OrbInt(t \Orb, .)
\]
\begin{lemme}[\cite{HarishDeBacker} lemme 5.2, \cite{WaldspurgerHomogeneite} 5.1]
	\label{lemme:homogoneite_integrales_orbitales}
	Il existe une normalisation des mesures invariantes telle que pour tout $t \in F^*$ et toute orbite nilpotente $\Orb \in \Orb(0)$, on ait les relations
	\[
		\rho(t) \OrbInt_{\g}(\Orb, .) = |t|^{-\frac{d(\Orb)}{2}} \OrbInt(t \Orb, .) \qquad \textrm{ et } \qquad \rho(t^2) \OrbInt_{\g}(\Orb, .) = |t|^{-d(\Orb)} \OrbInt_{\g}(\Orb, .)
	\]
\end{lemme}
\begin{proof}
	Voir \cite{HarishDeBacker} lemme 5.2 (noter que notre définition de $\rho(t)$ diffère de \cite{HarishDeBacker} d'où la différence de signe), \cite{WaldspurgerHomogeneite} 5.1 (dans lequel le signe est identique au nôtre). 
\end{proof}
Dans la suite on fixe une telle normalisation (le choix n'est pas unique). On note $\mathcal{N} \subset \g$ l'ensemble des matrices nilpotentes de $\g$ et $\distr(\g)^G(\mathcal{N})$ l'espaces des distributions supportées dans $\mathcal{N}$. Comme $t\mathcal{N} = \mathcal{N}$ pour tout $t \in F^*$, l'espace $\distr(\g)^G(\mathcal{N})$ est stable par $\rho$. Pour $d \in \N$, on pose
\[
	\distr(\g)^G(\mathcal{N})_d = \set{T \in \distr(\g)^G(\mathcal{N}), \, \forall t \in F^*, \, \rho(t^2) T = |t|^{-d}T} 
\]
On dira qu'un élément de $\distr(\g)^G(\mathcal{N})_d$ est homogène de degré $d$. D'après le lemme \ref{lemme:homogoneite_integrales_orbitales}, si $\Orb \in \Orb(0)$ est une orbite nilpotente, alors $\OrbInt_{\g}(\Orb, .)$ est homogène de degré $d(\Orb)$. Le lemme suivant assure que le degré d'homogénéité fournit une graduation de l'espace vectoriel $\distr(\g)^G(\mathcal{N})$.
\begin{lemme}
	\label{lemme:homogeneite_independance}
	\
	\begin{enumerate}
		\item La famille des distributions $(\OrbInt_{\g}(\Orb, .))_{\Orb \in \Orb(0)}$ forme une base de l'espace $\distr(\g)^G(\mathcal{N})$.
		\item On a la décomposition
		\[
			\distr(\g)^G(\mathcal{N}) = \bigoplus_{d \in \N} \distr(\g)^G(\mathcal{N})_d
		\]
		\item Pour $d \in \N$, la famille $(\OrbInt_{\g}(\Orb, .))_{\Orb \in \Orb(0)_d}$ forme une base de $\distr(\g)^G(\mathcal{N})_d$.
	\end{enumerate}
\end{lemme}
\begin{proof}
	\
	\begin{enumerate}
		\item Voir \cite{HarishDeBacker} lemme 5.1.
		\item D'après le point précédent, toute distribution dans $\distr(\g)^G(\mathcal{N})_d$ est comibinaison linéaire d'intégrales orbitales nilpotentes, en particulier c'est une combinaison linéaire de distributions homogènes. Montrons que la somme est directe. Soit donc $(T_d)_{d \in \N}$ une famille à support fini de distributions telles que $T_d \in \distr(\g)^G(\mathcal{N})_d$ pour tout $d \in \N$ et $\sum_{d \in \N} T_d = 0$. Pour tout $t \in F^*$, on a donc 
		\[
			0 = \rho(t^2)\pars{\sum_{d \in \N} T_d} = \sum_{d \in \N} |t|^{-d} T_d
		\]
		Par indépendance des caractères $(t \mapsto |t|^{-d})_{d \in \N}$, on déduit que les distributions $T_d$ sont toutes nulles.
		\item Ce dernier point est une conséquence des deux précédents.
	\end{enumerate}
\end{proof}
En particulier, le lemme entraine qu'une distribution $T \in \distr(\g)^G(\mathcal{N})$ homogène de degré $0$ est proportionnelle à l'intégrale orbitale $\OrbInt_{\g}(0, \cdot)$ en $0$. Terminons cette section par une petit lemme technique qui permet d'affaiblir la condition d'homogénéité.

\begin{lemme}
	\label{lemme:homogeniete_sous-groupe}
	Fixons $t_0 \in F^*$ un élément de $F$ tel que $|t_0| \ne 1$. Alors une distribution $T \in \distr(\g)^G(\mathcal{N})$ est homogène de degré $d$ si et seulement si on a
	\[
		\rho(t_0^{2}) T = |t_0|^{-d} T
	\]
\end{lemme}
\begin{proof}
	Le sens direct est clair. Supposons réciproquement $\rho(t_0^{2}) T = |t_0|^{-d} T$. Alors pour tout $k \in \Z$ on vérifie que $\rho(t_0^{2k}) T = \rho(t_0^{2})^k T = |t_0|^{-dk} T$. On décompose $T = \sum_{i} T_{i}$ où $T_i$ est homogène de degré $i$ (on utilise le lemme \ref{lemme:homogeneite_independance}), on a alors pour tout $k \in \Z$,
	\[
		 T = |t_0|^{dk} \rho(t_0^{2k}) T  = \sum_{i} |t_0|^{(d-i)k} T_i
	\]
	Or comme $|t_0| \ne 1$, les caractères $k \to |t_0|^{(d-i)k}$ de $\Z$ sont distincts donc indépendants. Et donc il reste $T = T_d$, ce qui conclut.
\end{proof}


\subsection{Distributions et exponentielle} 
\label{sub:distributions_exponentielle}

On rappelle que l'exponentielle est une application continue et $G$-equivariante $\exp : \lie{v} \to V$ où $\lie{v} \subset \g$ et $V \subset G$ sont des ouverts stablement invariants de $\g$ et $G$ respectivement. De manière duale, on dispose d'applications linéaires
\[
	\exp^* : \hecke(V) \to \hecke(\lie{v}) \qquad \qquad \exp^{**} : \distr(\lie{v}) \to \distr(V)
\]
définies par $\exp^*(f) = f \circ \exp$ et $\exp^{**}(D) = D \circ \exp^*$. On peut prolonger l'application $\exp^*$ sur $\hecke(G)$ tout entier ($\exp^*f$ ne dépend en fait que de $f_{|V}$), et la faire arriver dans $\hecke(\g)$ en prolongeant par $0$ hors de $\lie{v}$.
\[
	\exp^* : \hecke(G) \to \hecke(\g)
\]
De même, $\exp^{**}$ se prolonge à $\distr(\g)$ ($\exp^{**} d$ ne dépend en fait que de $d_{|\lie{v}}$) et on peut identifier $\distr(V)$ au sous-espace de $\distr(G)$ des distributions à support dans $V$.
\[
	\exp^{**} : \distr(\g) \to \distr(G)
\]
Comme $\exp$ est $G$-equivariante (i.e. vérifie la relation $\exp(g.X) = g.\exp(X)$) alors $\exp^*$ et $\exp^{**}$ le sont aussi. Donc une distribution $G$-invariante $D \in \distr(\g)^G$ à support dans $\lie{v}$ est amenée par $\exp^{**}$ sur une distribution $G$-invariante à support dans $V$. En particulier on dispose du lemme suivant.
\begin{lemme}
	\label{lemme:IntegraleOrbitaleExponentielle}
	Soit $X \in \g$. Si $X \in \lie{v}$, alors on a la relation
	\[
		\exp^{**}(\OrbInt_{\g}(X,.)) = \OrbInt_{G}(\exp(X),.)
	\]
	Sinon $\exp^{**}(\OrbInt_{\g}(X,.)) = 0$.
\end{lemme}
\begin{proof}
	Soit $f \in \hecke(G)$. Si $X \in \lie{v}$, alors
	\begin{align*}
		\OrbInt_{\g}(X,\exp^{*} f) &= \int_{G / C_G(X)} (\exp^*f)(g.X) \, \D x^* \\
		&= \int_{G / C_G(X)} f(\exp(g.X)) \, \D x^* \\
		&= \int_{G / C_G(X)} f(g \exp(X) g^{-1}) \, \D x^* \\
		&= \int_{G / C_G(\exp(X))} f(g \exp(X) g^{-1}) \, \D x^* \\
		&= \OrbInt_{G}(\exp(X),f)
	\end{align*}
	On a l'égalité $C_G(X) = C_G(\exp(X))$ vient de l'injectivité de l'exponentielle (l'équivariance donnant l'inclusion $C_G(X) \subset C_G(\exp(X))$). Si $X$ n'est pas dans $\lie{v}$, alors les support de $\exp^*f$ est nulle sur toute la classe de conjugaison de $X$, donc $\exp^{**}(\OrbInt_{\g}(X,.)) = 0$.
\end{proof}
\begin{lemme}
	\label{lemme:IntegraleOrbitaleStableExponentielle}
	Soit $X \in \g$. Si $X \in \lie{v}$, alors on a la relation
	\[
		\exp^{**}(\OrbInt_{\g}^{\st}(X,.)) = \OrbInt_{G}^{\st}(\exp(X),.)
	\]
	Sinon $\exp^{**}(\OrbInt_{\g}^{\st}(X,.)) = 0$.
\end{lemme}
\begin{proof}
	Il suffit de remarquer que $\lie{v}$ est stable par conjugaison stable et que $X, X' \in \lie{v}$ sont stablement conjugués si et seulement si $\exp(X)$ et $\exp(X')$ le sont.
\end{proof}

On définit de manière analogue $\ln^* : \hecke(\lie{v}) \to \hecke(V)$ et $\ln^{**} : \distr(V) \to \distr(\lie{v})$. Ces applications sont inverses de $\exp^*$ et $\exp^{**}$ respectivement. On a donc pour tout $X \in \lie{v}$
\[
	\ln^{**} \OrbInt_{G}(\exp(X),.) = \OrbInt_{\g}(X,.)
\]

%
%

\subsection{Descente des intégrales orbitales aux algèbres de Lie} 
\label{sub:descente_integrales_orbitales}

Le lemme \ref{lemme:IntegraleOrbitaleStableExponentielle} permet de ramener les intégrales orbitales dans $G$ au voisinage de l'identité en des intégrales orbitales sur l'algèbre de Lie. Si $\gamma \in G$ est un élément semi-simple, on note $G_{\gamma} = C_G(\gamma)^0$ la composante connexe du centralisateur de $\gamma$ dans $G$ et $\g_{\gamma}$ son algèbre de Lie. On considère l'application $\exp_{\gamma} : X \mapsto \exp(X) \gamma$, elle est $G_{\gamma}$-equivariante.

On reproduit ici une lemme de \cite{WaldspurgerEndoscopieTordue} permettant de descendre les intégrales orbitales tordues sur le groupe $G$ au voisinage d'un point $\gamma$ quelconque à celle sur l'algèbre de Lie $\g_{\gamma}$ via l'exponentielle.

\begin{lemme}[cf \cite{WaldspurgerEndoscopieTordue} 2.4]
	\label{lemme:descente}
	Il existe un voisinage $\lie{U}$ de $0$ dans $\g_{\gamma}$ tel que
	\begin{enumerate}
		\item Pour tout $f \in \hecke\pars{\widetilde{G^0}}$, il existe $\varphi \in \hecke\pars{\g}$ tel que pour tout $X \in \lie{U}$, on a
		\[
			\OrbInt_{\widetilde{G^0}}(\exp(X) \gamma, f) = \OrbInt_{\g_{\gamma}}(X, \varphi)
		\]
		 \item L'ouvert $\lie{U} \subset \g_{\gamma}$ est un $G_{\gamma}$-domaine qui vérifie
			\begin{enumerate}
				\item Pour $X \in \lie{U}$, $X$ est semi-simple régulier dans $\g_{\gamma}$ si et seulement si $\exp(X) \gamma$ l'est dans $\widetilde{G^0}$.
				\item Si $X,Y \in \lie{U}$ et $x \in G$ sont tels que $x.[\exp(X)\gamma] = \exp(Y)\gamma$, alors $x \in C_G(\gamma)$. A fortiori, pour $X \in \lie{U}$, on a l'inclusion $C_G(\exp(X) \gamma)^0 \subset C_G(\gamma)^0$
			\end{enumerate}
	\end{enumerate}
	 
\end{lemme}
\begin{proof}
	Il s'agit du résultat de \cite{WaldspurgerEndoscopieTordue} 2.4.
\end{proof}

Notons $\adherenceIntOrb(\g_{\gamma})$ le sous-espace des fonctions $f \in \hecke(\g_{\gamma})$ telles que $\OrbInt_{\g_{\gamma}}(X,f) = 0$ pour tout $X \in \g_{\gamma}$ semi-simple régulier.
Le choix de $\varphi \in \hecke(\g_{\gamma})$ dans le lemme n'est pas unique a priori, mais sa classe modulo $\adherenceIntOrb(\g_{\gamma})$ l'est. On notera $\exp_{\gamma, G}^{*} f \in \hecke(\g_{\gamma})/\adherenceIntOrb(\g_{\gamma})$ cette classe. Si $d \in \distr^{G_{\gamma}}(\g_{\gamma})$ est une distribution invariante, elle est nulle sur $\adherenceIntOrb(\g_{\gamma})$. On peut alors définir $\exp_{\gamma, G}^{**} d \in \distr(G)$ par $\exp_{\gamma, G}^{**} d(f) = d(\exp_{\gamma, G}^{*} f)$ (il n'y a pas d'ambiguité puisque $d$ est nulle sur $\adherenceIntOrb(\g_{\gamma})$). Par construction, on a $\exp_{\gamma, G}^{**} \OrbInt_{\g_{\gamma}}(X, \cdot) = \OrbInt_G(\exp(X)\gamma, \cdot)$ pour $X$ dans un voisinage de $0$ comme dans le lemme \ref{lemme:descente}.

%
%
%




\section{Endoscopie pour les algèbres de Lie} 
\label{sec:endoscopie_pour_les_algebres_de_lie}

\subsection{Intégrales orbitales} 
\label{sub:integrales_orbitales}

Si $X \in g$, on définit son orbite, notée $X^G$ par

\[
	X^G = \set{X' \in g, \, \exists x \in G, \, \Ad(x).X = X'}
\]
On munit $X^G$ d'une mesure invariante par adjonction. L'intégrale orbitale associée est la distribution $\OrbInt_{\g}(X,.)$ définie pour $f \in \hecke(\g)$ par
\[
	\OrbInt_{\g}(X,f) = \int_{X^G} f(X') \, \D X'
\]
On vérifie que $\OrbInt_{\g}(X,.) \in \distr^G(\g)$.

On note $g_{\reg}$ le sous-ensemble de $g$ formé par les éléments semi-simples réguliers. Pour $X \in \alg{\g}_{\reg}$, on définit son orbite stable par
\[
	X^{G, \st} = \set{X' \in g, \, \exists x \in \alg{\g}(\overline{F}), \, \Ad(x).X = X'}
\]
L'orbite stable est une union disjointe d'orbites ordinaires, fixons $\mathcal{E}(X) \subset g$ un ensemble de représentants des orbites contenues dans $X^{G, \st}$. On a la décomposition en union disjointe
\[
	X^{G, \st} = \bigsqcup_{X' \in \mathcal{E}(X)} (X')^{G}
\]
On sait que $\mathcal{E}(X)$ est un ensemble fini via la cohomologie galoisienne (voir \cite{KottwitzRationalConjugacyClasses} pour plus de détails).
On définit l'intégrale orbitale stable $\OrbInt_{\g}^{\st}(X,.)$ par
\[
	\OrbInt_{\g}^{\st}(X,.) = \sum_{X' \in \mathcal{E}(X)} \OrbInt_{\g}(X',.)
\]
moyennant un choix convenable de normalisation pour les mesures. On note $\distr(\g)^{\st}$ la clôture dans $\distr(\g)$, pour la topologie faible, du sous-espace engendré par $\set{\OrbInt_{\g}^{\st}(X,.), \, X \in g_{\reg}}$. C'est-à-dire que $\distr(\g)^{\st}$ est le sous-espace des distribution $D \in \distr(\g)$ qui sont nulles sur toutes les fonctions $f \in \hecke(\g)$ telles $\OrbInt_{G}^{\st}(X,f) = 0$ pour tout $X \in g_{\reg}$.


\subsection{Facteur de transfert} 
\label{sub:facteur_de_transfert}

On se donne $\textbf{\textit{H}}$ un groupe apparaissant dans une donnée endoscopique de $\textbf{\textit{G}}$ (
dans la suite on sera intéressé par le cas $\alg{\widetilde{G^0}} = \GL_{2n} \theta \rangle$ et $\alg{H} = \SO_{2n+1}$). Un tel groupe $\textbf{\textit{H}}$ est algébrique, défini sur $F$, réductif, connexe et quasi déployé. Sa dimension est inférieure ou égale à celle de $\textbf{\textit{G}}$, mais son rang est égal à celui de $\textbf{\textit{G}}$. 
Langlands et Shelstad ont défini :
\begin{enumerate}
	\item Un sous-ensemble $\h_{G-\reg}$ de $\h_{\reg}$ qui est un ouvert de Zariski dense dans $\h$.
	\item Une correspondance entre orbites stables dans $\h_{G-\reg}$ et celles dans $\g_{\reg}$.
	\item Une application (définie à un scalaire près), appelée facteur de transfert,
	\[
		\Delta_{\g,\h} : \h_{G-\reg} \times \g_{\reg} \longrightarrow \C
	\]
	telle que pour $(Y,X) \in \h_{G-\reg} \times \g_{\reg}$, on ait :
	\begin{enumerate}
		\item Si $\Delta_{\g,\h}(Y,X) \ne 0$, alors $\Orb^{\st}(Y)$ et $\Orb^{\st}(X)$ se correspondent.
		\item Si $Y' \in \Orb^{\st}(Y)$ et $X' \in \Orb^{\st}(X)$, alors
		\[
			\Delta_{\g,\h}(Y,X) = \Delta_{\g,\h}(Y',X')
		\]
	\end{enumerate}
\end{enumerate}
Si $(Y,X) \in \h_{G-\reg} \times \g_{\reg}$ sont tels que $\Orb^{\st}(Y)$ et $\Orb^{\st}(X)$ se correspondent, on pose
\[
	\OrbInt_{\g,\h}(Y,.) = \sum_{X' \in \mathcal{E}(X)} \Delta_{\g,\h}(Y,X') \, \OrbInt_{\g}(X', .)
\]
On remarque que comme $\OrbInt_{\g,\h}(Y,X') = 0$ si $\Orb^{\st}(Y)$ et $\Orb^{\st}(X)$ ne se correspondent pas, on peut d'ailleurs réécrire si on veut
\[
	\OrbInt_{\g,\h}(Y,.) = \sum_{X \in \g_{\reg}/\text{conj}} \Delta_{\g,\h}(Y,X) \, \OrbInt_{\g}(X, .)
\]
où $X$ parcourt l'ensemble des classes de conjugaison de $\g_{\reg}$.

\subsection{Transfert des fonctions et distributions} 
\label{sub:transfert_des_fonctions_et_distributions}

\begin{defn}[Transfert des fonctions]
	\label{def:transfertLie}
	Soient $f \in \hecke(\g)$, $f^H \in \hecke(\h)$. On dit que $f^H$ est un transfert de $f$ si et seulement si pour tout $Y \in \h_{G-\reg}$, on a l'égalité
	\[
		\OrbInt_{\g,\h}(Y,f) = \OrbInt_{\h}^{\st}(Y,f^H)
	\]
	On dit alors que les intégrales orbitales (semi-simples) de $f$ et de $f^H$ se correspondent.
\end{defn}

On sait que le transfert des fonctions existe (il s'agit d'un résultat difficile, voir notamment \cite{NgoLemmeFondamental}, \cite{WaldspurgerCaracteristique} et \cite{WaldspurgerTransfert}). Remarquons que le transfert n'est pas unique. En fait, si l'on note $\transfertZero(\h)$ le sous-espace de $\hecke(\h)$ des fonctions $f^H$ telles que $\OrbInt_{\h}^{\st}(Y,f^H) = 0$ pour tout $Y \in h_{G-\reg}$ (c'est-à-dire $f^H$ est un transfert de la fonction nulle), alors l'ensemble des transferts d'une fonction $f \in \hecke(\g)$ est une classe modulo $\transfertZero(\h)$, ce qui définit une fonction

\[
	\operatorname{Transfert} : \hecke(\g) \longrightarrow \hecke(\h)/\transfertZero(\h)
\]

De manière duale au transfert sur les fonctions, on peut définir une notion de transfert sur les distributions de $\h$ vers $\g$.

\begin{defn}[Transfert des distributions]
	\label{def:transfertDistributionsLie}
	soient $D^H \in \distr(\h)^H$ et $D \in \distr(\g)^G$. On dit que $D$ est un transfert de $D^H$ si et seulement si pour toutes $f \in \hecke$ $f^H \in \hecke(\h)$ telles que $f^H$ soit un transfert de $f$, on a l'égalité
	\[
		D(f) = D^H(f^H)
	\]
\end{defn}

Si la distribution $D^H \in \distr(\h)$ admet un transfert, alors elle est nulle sur $\transfertZero(\h)$, donc est stablement invariante.

D'après la définition, si $Y \in h_{G-\reg}$, le transfert de la distribution $\OrbInt_{\h}^{\st}(Y, .) \in \distr(\h)^{H, \st}$ est $\OrbInt_{G,H}(Y,.) \in \distr(\g)^G$. Le transfert des intégrales orbitales (stables) nilpotentes est plus difficile à décrire a priori (cf. \cite{WaldspurgerEndoscopie}).


\subsection{Endoscopie non standard} 
\label{sub:endoscopie_non_standard}

On peut définir une notion analogue de transfert sur les algèbres de Lie pour certaines paires de groupes qui ne sont pas des données endoscopiques au sens ordinaire, on parle d'endoscopie non standard. Donnons une brève description de ce qu'est un triplet endoscopique non standard $(G_1, G_2, j_*)$ (on se réfèrera à \cite{WaldspurgerEndoscopieTordue} 1.7 pour une définition précise et une classification de ces triplets). Dans un tel triplet, $G_1$, $G_2$ sont des groupes réductifs connexes et quasi-déployés sur $F$. Pour $i = 1,2$, fixons un tore maximal $T_i$ défini sur $F$ d'un sous-groupe de Borel de $G_i$ défini sur $F$. Notons $\Omega_i$ le groupe de Weyl de $G_i$ relativement à $T_i$. Enfin $j$ est la donnée d'un isomorphisme $j_* : X_*(T_1) \otimes \Q \to X_*(T_2) \otimes \Q$ équivariant pour l'action de $\Gal(\overline{F}/F)$ (où désigne le groupe des cocaractères de $T_i$), et d'un isomorphisme $j_{\Omega} : \Omega_1 \to \Omega_2$ tel que pour tous $t_1 \in T_1$, $\omega_1 \in \Omega_1$, on ait

\[
	j_{\Omega}(\omega_1) \circ j_* = j_* \circ \omega_1
\]
Ces données vérifiant certaines conditions précisées dans \cite{WaldspurgerEndoscopieTordue} 1.7.
En passant aux algèbre de Lie,
l'application $j_*$ induit (en tensorisant par $\overline{F}$) un isomorphisme entre les algèbres de Lie $\mathfrak{t}_1(F)$ et $\mathfrak{t_2}(F)$
des tores compatibles à l'action des groupes de Weyl,
et donc une bijection 
\[
	(\mathfrak{t}_1(F)/\Omega_1)^{\Gal(\overline{F}/F)} \sim (\mathfrak{t}_2(F)/\Omega_2)^{\Gal(\overline{F}/F)}
\]
donc entre les classes de conjugaison stables semi-simples dans $\g_1(F)$ et $\g_2(F)$. On peut alors définir une notion de transfert entre les fonctions de $\hecke(\g_1(F))$ et $\hecke(\g_2(F))$ (le facteur de transfert vaut $1$) puis entre les distributions stables sur $\g_2$ et $\g_1$. Toutefois, à la différence de l'endoscopie standard, il n'y a pas de transfert au niveau des groupes $G_1$ et $G_2$.
Un exemple important
d'endoscopie non standard est le cas dans lequel $G_1 = \Sp(2n)$ est le groupe symplectique et $G_2 = \SO(2n+1)$ est le groupe spécial orthogonal.

La correspondance entre classes de conjugaisons stables semi-simples régulières de $\lie{sp}_{2n}(F)$ et $\so_{2n+1}(F)$ se décrit explicitement de la manière suivante. Si $X \in \lie{sp}_{2n}(F)$ est semi-simple régulier, on note $\Lambda(X)$ l'ensemble de ses valeurs propres. De même, si $Y \in \so_{2n+1}(F)$ est semi-simple régulier, on note $\Lambda(Y)$ l'ensemble de ses valeurs propres. Alors les classes de conjugaisons stables de $X \in \lie{sp}_{2n}(F)$ et $Y \in \so_{2n+1}(F)$ se correspondent si et seulement si $\Lambda(Y) = \Lambda(X) \cup \set{0}$.


\begin{defn}[Transfert des fonctions]
	\label{def:transfertNonStandardLie}
	Soient $f_1 \in \hecke(\g_1)$, $f_2 \in \hecke(\g_2)$. On dit que $f_2$ est un transfert de $f_1$ si pour tout $X_1 \in \g_1$ et $X_2 \in \g_2$ dont les classes de conjugaison stables se correspondent, on a l'égalité
	\[
		\OrbInt_{\g_1}^{\st}(X_1,f_1) = \OrbInt_{\g_2}^{\st}(X_2,f_2)
	\]
\end{defn}
De même que précédemment, le transfert définit une application
\[
	\operatorname{Transfert} : \hecke(\g_1)/\transfertZero(\g_1) \longrightarrow \hecke(\g_2)/\transfertZero(\g_2)
\]
L'existence du transfert endoscopique non standard a elle aussi été établie grâce aux travaux de Ngo Bao Chau sur le lemme fondamental (cf. \cite{NgoLemmeFondamental}) et de Waldspurger qui montrent que le lemme fondamental non standard implique le transfert non standard (cf. \cite{WaldspurgerEndoscopieTordue} ou \cite{WaldspurgerGLnTordu}).





\section{Homogénéité et transfert} 
\label{sec:homogeneite_et_transfert}

\subsection{Homogénéité et endoscopie standard} 
\label{sub:homogeneite_endoscopie}

On suppose que $G$ est quasi-déployé sur $F$. On définit un caractère $\chi_{G,H} : F^{\times} \to \C^*$ d'ordre $\le 2$ du groupe multiplicatif $F^{\times}$ de la manière suivante. Soit $T_G \subset G$ un tore de $G$ défini sur $F$ contenu dans un sous-groupe de Borel de $G$ défini sur $F$ et $T_H \subset H$ le tore de $H$ correspondant ($G$ et $H$ ont même rang). Désignons par $\sigma_G$ et $\sigma_H$ les actions de $\sigma \in \Gal(\overline{F}/F)$ respectivement sur $T_G$ et $T_H$ (que l'on identifie). Ces deux actions diffèrent d'un élément du groupe de Weyl de $G$, c'est-à-dire que pour tout $\sigma \in \Gal(\overline{F}/F)$, il existe un unique élément $\omega(\sigma) \in \Omega(G,T_G)$ tel que $\sigma_H = \omega(\sigma) \sigma_G$. On définit alors $\tilde{\chi}_{G,H} : \Gal(\overline{F}/F) \to \C^*$ par

\[
	\tilde{\chi}_{G,H}(\sigma) = \epsilon(\omega(\sigma))
\]
Où $\epsilon$ est l'homomorphisme de signature sur le groupe de Weyl. Enfin, par la théorie du corps de classe local, le caractère $\tilde{\chi}_{G,H} : \Gal(\overline{F}/F) \to \C^*$ correspond à un caractère $\chi_{G,H} : F^{\times} \to \C^*$. Remarquons que comme $\tilde{\chi}_{G,H}$ est une signature, alors $\tilde{\chi}_{G,H}^2 = 1$ et donc $\chi_{G,H}^2 = 1$. On pose $d_{G,H} = \dim(G)-\dim(H)$ et
\begin{equation}
	\psi_{G,H}(t) = \chi_{G,H}(t) \, |t|_F^{d_{G,H}/2}
\end{equation}
On vérifie alors que $\psi_{G,H}(t^2) = |t|_F^{d_{G,H}}$. On reproduit ici un lemme de \cite{Ferrari} sur l'homogénéité du facteur de transfert pour les algèbre de Lie (nous avons pris des notations un peu différentes pour des raisons pratiques).

\begin{lemme}[\cite{Ferrari} lemme 3.2.1]
	
	Soient $X \in \g_{\textrm{reg}}$ un élément semi-simple régulier de $\g$, $Y \in \h_{G-\textrm{reg}}$ un élément semi-simple régulier de $\h$ et $t \in F^{\times}$ un scalaire. On a la relation d'homogénéité
	\begin{equation}
		\Delta_{\g,\h}(tX, tY) = \psi_{G,H}(t) \, \Delta_{\g,\h}(X,Y)
	\end{equation}
\end{lemme}
\begin{proof}
	La preuve est donnée dans \cite{Ferrari} lemme 3.2.1.
\end{proof}
\begin{cor}[\cite{Ferrari} proposition 3.2.2]
	\label{cor:homogeneite_transfert_fonctions}
	Soient $f \in \hecke(G)$ et $f^H \in \hecke(H)$, et $t \in F^{\times}$. On suppose que $f^H$ est un transfert de $f$, alors la fonction $\psi_{G,H}^{-1}(t) (f^H)_t$ est un transfert de $f_t$.
\end{cor}
\begin{proof}
	La preuve est donnée dans \cite{Ferrari} proposition 3.2.2.
\end{proof}
En particulier, dans les notations du lemme, la fonction $|t|^{-d_{G,H}} (f^H)_{t^2}$ est un transfert de $f_{t^2}$. Signalons que Shahidi formule un résultat similaire dans \cite{ShahidiLanglandsConjecture} lemme 9.7 (le différence étant qu'il formule le résultat dans les groupes, en définissant la dilatation $f_t$ par transport de structure avec l'exponentielle).
\begin{cor}
	\label{cor:homogeneite_transfert_distributions}
	Soit $D_G \in \hecke(G)^*$ et $D_H \in \hecke(H)^*$ des distributions sur $G$ et $H$ respectivement. On suppose que $D_G$ est un transfert de $D_H$. Alors
	\begin{enumerate}
		\item La distribution $\psi_{G,H}(t) \, \rho(t)D_G$ est un transfert de $\rho(t) D_H$.
		\item Si $D_H$ est homogène de degré $d_H$ alors $D_G$ est homogène de degré $d_G$ avec
		\[
			d_{G,H} = d_G - d_H
		\]
	\end{enumerate}

\end{cor}
\begin{proof}
	\ \\
	\begin{enumerate}
		\item Soient $f \in \hecke(G)$ et $f^H \in \hecke(H)$ tels que $f^H$ est un transfert de $f$. On a alors
		\begin{align*}
			[\psi_{G,H}(t) \, \rho(t)D_G](f) &= D_G(\psi_{G,H}(t) \, f_t) \\
			&= D_H( (f^H)_t ) \\
			&= [\rho(t) D_H](f^H)
		\end{align*}
		Ce qui montre le premier point.
		\item Supposons $D_H$ est homogène de degré $d_H$. Soit $f \in \hecke(G)$. On se donne $f^H \in \hecke(H)$ un transfert de $f$. Alors on a
		\begin{align*}
			D_G(f_{t^2}) &= D_H (|t|^{-d_{G,H}} (f^H)_{t^2}) \\
			&= |t|^{-d_{G,H}} D_H ((f^H)_{t^2}) \\
			&= |t|^{-d_{G,H}-d_H} D_H(f^H) \\
			&= |t|^{-d_{G,H}-d_H} D_G(f)
		\end{align*}
		Ce qui signifie que $D_G$ est homogène de degré $d_G = d_H + d_{G,H}$.
	\end{enumerate}
\end{proof}


\subsection{Homogénéité et endoscopie non standard} 
\label{sub:homogeneite_endoscopie_non_standard}

On se donne $(G_1, G_2, j)$ un triplet endoscopique non standard. Comme le facteur de transfert vaut identiquement $1$ (donc est homogène de degré $0$), on déduit un résultat analogue d'homogénéité.

\begin{cor}
	\label{cor:homogeneite_transfert_distributions_non_standard}
	Soit $D_1 \in \hecke(G_1)^*$ et $D_2 \in \hecke(G_2)^*$ des distributions sur $G_1$ et $G_2$ respectivement. On suppose que $D_1$ est un transfert de $D_2$. Alors
	\begin{enumerate}
		\item La distribution $\rho(t)D_1$ est un transfert de $\rho(t) D_2$.
		\item Si $D_2$ est homogène de degré $d$ alors $D_1$ aussi.
	\end{enumerate}
	
\end{cor}
\begin{proof}
	La preuve est la même qu'en \ref{cor:homogeneite_transfert_distributions} en plus simple.
\end{proof}

\begin{cor}
	Il existe $\lambda \in \C$ tel que la distribution $\lambda \OrbInt_{\g_1}(0, .)$ soit le transfert de l'intégrale orbitale $\OrbInt_{\g_2}^{\st}(0, .) = \OrbInt_{\g_2}(0, .)$.
	En outre, si $p$ est suffisamment grand, la constante $\lambda$ est strictement positive.
\end{cor}
\begin{proof}
	La distribution $\OrbInt_{\g_2}^{\st}(0, .)$ est homogène de degré $0$, donc d'après le corollaire \ref{cor:homogeneite_transfert_distributions_non_standard}, son transfert l'est aussi. Comme le transfert d'une distribution à support nilpotent est à support nilpotent, d'après le lemme \ref{lemme:homogeneite_independance}, ce transfert s'écrit donc $\lambda \OrbInt_{\g_1}(0, .)$ avec $\lambda \in \C$.
	
	Pour calculer la constante $\lambda$, il suffit de connaitre explicitement le transfert pour un coupe donné de fonctions. C'est en particulier l'objet du lemme fondamental non standard, valable pour $p$ suffisamment grand. Pour $i \in \{1,2\}$, si $f_i$ est la fonction caractéristique d'un ``réseau hyperspécial'' de $\g_i$ alors celui-ci affirme que $cf_2$ est un transfert de $f_1$ pour $c > 0$ une constante qui dépend du choix des mesures (voir \cite{WaldspurgerEndoscopieTordue} pour les définitions précises). En particulier, on tire donc $\lambda = c$.
\end{proof}

\begin{rem}
	Waldspurger nous a suggéré une autre méthode pour le calcul de la constante à partir des expressions des Germes de Shalika.
\end{rem}



\section{Endoscopie pour les groupes} 
\label{sec:endoscopie_pour_les_groupes}

On définit de manière analogue une notion de transfert endoscopique pour les groupes (en fait cette notion est antérieure historiquement). On définit notamment un facteur de transfert $\Delta_{G,H}$ pour les groupes $G$ et $H$ (voir la définition dans \cite{LSTransferFactor}) et dans le cadre plus général de l'endoscopie tordue (voir \cite{KottwitzShelstad}) qui fait correspondre les classes de conjugaison stables d'éléments semi-simples réguliers de $G$ et $H$, et les intégrales orbitales stables et endoscopiques. Cela permet de définir le transfert de manière analogue $f^{H} \in \Hecke(H)$ est un transfert de $f \in \Hecke(G)$ si les intégrales orbitales stables semi-simple régulières de $f^{H}$ sont égales aux intégrales ``endoscopiques'' correspondantes de $f$. C'est-à-dire que pour $h$ l'on pose
\[
	\OrbInt_{H}^{\st}(h,.) = \sum_{h' \in \mathcal{E}(h)} \OrbInt_{H}(h', .)
\]
et
\[
	\OrbInt_{G, H}(h,.) = \sum_{g} \Delta_{G,H}(h,g) \, \OrbInt_{G}(g, .)
\]
Et que l'on dit que $f^H$ est transfert de $f$ si $\OrbInt_{H}^{\st}(h,f^H) = \OrbInt_{G, H}(h,f)$ pour tout $h \in H$ suffisamment régulier. De manière duale, les distributions stables sur $H$ se transfèrent à $G$ comme sur les algèbres de Lie.

\subsection{Les groupes $\SO(2n+1)$ et $\GL_{2n}$ tordu} 
\label{sub:les_groupes_so_2n_1_et_gl_2n_tordu}

Dans le cas qui nous intéresse en particulier, le groupe endoscopique principal de $\widetilde{G^0} = \GL_{2n}(F) \theta$ est $G' = \SO_{2n+1}(F)$ (voir \cite{ArthurTraceFormula2} paragraphe 9). Ici le facteur de transfert est trivial (i.e. égal à $1$), le calcul est fait dans \cite{WaldspurgerFormulaire} 1.11.

Donc les fonctions sur $\widetilde{G^0}$ se transfèrent à $G'$, et de manière duale, les distributions stables sur $G'$ se transfèrent à $\widetilde{G^0}$.

La correspondance entre classes de conjugaisons stables semi-simples régulières de $\widetilde{G^0}$ et $G'$ est décrite explicitement dans \cite{WaldspurgerGLnTordu} par l'application norme $\widetilde{G^0}_{\textrm{reg}}/\st \to G'_{\textrm{reg}}/\st$. Rappelons sa définition. Si $g\theta \in \widetilde{G^0}$ est semi-simple régulier, on note $\Lambda(g\theta)$ l'ensemble des valeurs propres de $\theta(g)g \in \GL_{2n}(F)$ (qui sont toutes distinctes). De même, si $h \in \SO(2n+1)$ est semi-simple régulier, on note $\Lambda(h)$ l'ensemble des valeurs propres de $h$ (qui sont elles aussi distinctes et parmi lesquelles il y a $1$).
La norme de la classe de conjugaison stable de $g\theta$ est la classe de conjugaison stable de $h$ si et seulement si $\Lambda(h) = \set{-x, x \in \Lambda(g\theta)} \cup \set{1}$.


\subsection{Transfert et exponentielle} 
\label{sub:transfert_et_exponentielle}

Nous allons montrer une relation de commutation du transfert à l'exponentielle. On pose $\gamma = \mat{0 & I_n\\-I_n & 0}$ et $\eta = \gamma \theta \in \widetilde{G^0}$ dont le centralisateur est le groupe symplectique.

\begin{lemme}
	\label{lemme:normeExponentielle}
	Soient $X \in \h = \so(2n+1)$ et $Y \in \g_{\eta} = \lie{sp}(2n)$ des éléments semi-simples réguliers. Alors les propositions suivantes sont équivalentes :
	\begin{enumerate}
		\item Les classes de conjugaison stables de $X$ et $Y$ se correspondent.
		\item Il existe $r > 0$ tel que pour $t \in F^*$ tel que $|t| \le r$, les classes de conjugaison stables de $\exp(2 t X) \in \SO(2n+1)$ et de $\exp(t Y)\eta \in \GL(2n) \theta$ se correspondent.
		\item Il existe $t \in F^*$ tel que les classes de conjugaison stables de $\exp(2 t X) \in \SO(2n+1)$ et de $\exp(t Y)\eta \in \GL(2n) \theta$ se correspondent.
	\end{enumerate}
\end{lemme}
\begin{proof}
	Il s'agit de déterminer les relations entre les spectres respectifs. Soient $X \in \h = \so(2n+1)$ et $Y \in \g_{\eta} = \lie{sp}(2n)$ des éléments semi-simples réguliers. On se donne $t \in F^*$ suffisamment petit pour que l'exponentielle soit définie en $2 t X$ et $tY$. On a alors
	\[
		\Lambda(\exp(2tX)) = \exp(2t\Lambda(X))
	\]
	et
	\begin{align*}
		\Lambda(\exp(tY)\eta) &= \Lambda( \theta(\exp(tY)\gamma) \exp(tY) \gamma) \\
		&= \Lambda( \theta(\gamma)\exp(tY) \exp(tY) \gamma) \\
		&= \Lambda( \gamma\exp(2tY)\gamma) \\
		&= \Lambda( -\gamma\exp(2tY)\gamma^{-1}) \\
		&= \Lambda( -\exp(2tY)) \\
		&= -\exp(2t \Lambda(Y))
	\end{align*}
	Or il est clair que l'on a équivalence entre
	\begin{enumerate}
		\item $\Lambda(X) = \Lambda(Y) \cup \{0\}$
		\item $\exists t \in F^*, \, \exp(2t\Lambda(X)) = -\exp(2t \Lambda(Y)) \cup \{1\}$
		\item $\exists r > 0, \forall t \in F^*,\, |t| \le r,\, \exp(2t\Lambda(X)) = -\exp(2t \Lambda(Y)) \cup \{1\}$
	\end{enumerate}
\end{proof}

En particulier si les éléments semi-simples réguliers $X \in \h = \so(2n+1)$ et $Y \in \g_{\eta} = \lie{sp}(2n)$ se correspondent et sont dans les ouverts respectifs de convergence de l'exponentielle (s'ils se correspondent, ils sont simultanément dans le domaine de convergence) alors $\exp(2X)$ et $\exp(Y)\eta$se correspondent. De manière duale, on tire le lemme suivant.

\begin{lemme}
	\label{lemme:commutation_exponentielle_transfert_fonctions}
	Soit $f^H \in \hecke(H)$ est un transfert de $f \in \hecke(G)$.
	Soient $\varphi \in \hecke(\g_{\eta})$ et $\lie{U} \subset \g_{\eta}$ comme dans le lemme \ref{lemme:descente}, 
	alors pour tout $Y \in \lie{U}$, on a
	\[
		\OrbInt_{\h}^{\st}(X, \exp_2^* f^H) = \OrbInt_{\g_{\eta}}^{\st}(Y, \varphi)
	\]
	où l'on a noté $\exp_2(X) = \exp(2X)$.
\end{lemme}
\begin{proof}
	Si $Y \in \lie{U}$ alors l'exponentielle converge en $2X$ et on a
	\begin{align*}
		\OrbInt_{\h}^{\st}(X, \exp_2^* f^H) &= \OrbInt_{H}^{\st}(\exp(2X), f^H) \\
			&= \OrbInt_{G}^{\st}(\exp(Y)\eta, f) \\
			&= \OrbInt_{\g_{\eta}}^{\st}(Y, \varphi)
	\end{align*}

	%
\end{proof}
En d'autres termes, le lemme signifie que $\exp_2^* f^H$ est un transfert local de $\exp_{\eta, G}^* f$ au voisinage de $0$. Cela permet d'obtenir le lemme suivant.
\begin{cor}
	\label{cor:commutation_exponentielle_transfert_distributions}
	Si $d^G \in \distr(\g_{\eta})$ est un transfert de $d^H \in \distr(\h)$ avec $\Supp(d^G) \subset \lie{U}$, alors $\exp_{\eta, G}^{**}d^G$ est un transfert de $\exp_2^{**}d^H$.
\end{cor}
\begin{proof}
	Remarquons que $\exp_{\eta, G}^{**}d^G$ est bien défini car $d^G$ est stable.
	On se donne $f$ et $f^H$ telles que $f^H$ est transfert de $f$. D'après le lemme \ref{lemme:commutation_exponentielle_transfert_fonctions}, la fonction $\exp_2^* f^H$ coïncide avec un transfert de $\exp_{\eta, G}^* f$ sur le support de $d_H$, donc
	\begin{align*}
		\exp_{\eta, G}^{**}d^G(f) &= d^G(\exp_{\eta, G}^* f) \\
		&= d^H(\exp_2^* f^H) \\
		&= \exp_2^{**}d^H(f^H)
	\end{align*}
\end{proof}

\subsection{Transfert des intégrales orbitales} 
\label{sub:transfert_des_integrales_orbitales}

L'intégrale orbitale $\OrbInt_{G'}(1, \cdot)$ associée à $1 \in G'$ (ce qui correspond au Dirac en $1$) est stable d'après un résultat de Kottwitz donc peut se transférer. La proposition suivante est également démontrée dans \cite{Shahidi} prop 7.3.
\begin{prop}
	\label{prop:transfertIntegraleOrbitale}
	Il existe $\lambda > 0$ tel que $\lambda \OrbInt_{\widetilde{G^0}}(\eta, \cdot)$ soit le transfert de l'intégrale orbitale $\OrbInt_{G'}(1, \cdot)$.
\end{prop}
\begin{proof}
	D'après le corollaire \ref{cor:homogeneite_transfert_distributions_non_standard}, le transfert de la distribution $\OrbInt_{\g'}(0, \cdot)$ est homogène de degré $0$ donc est de la forme $\lambda \OrbInt_{\g}(0, \cdot)$ avec $\lambda > 0$, et donc d'après le corollaire \ref{cor:commutation_exponentielle_transfert_distributions}, le transfert de $\OrbInt_{G'}(1, \cdot)$ est $\lambda \OrbInt_{\widetilde{G^0}}(\eta, \cdot)$.
\end{proof}

\subsection{Transfert des caractères} 
\label{sub:transfert_des_caracteres}

Dans \cite{Arthur}, J. Arthur définit, pour les groupes classiques, les $L$-paquets de représentations tempérées et donc en particulier ceux pour les groupes $\SO(2n+1,F)$, $n \in \N$. Plus précisément, il prouve, traduit en termes de $G'$, le résultat suivant (qui dépend de la stabilisation de la formules des traces tordues et de résultats annexes qui sont en cours de rédaction par le groupe de travail Marseille-Paris):

\null

\begin{thm}[ \cite{Arthur} 1.5.1]
	\label{thm:arthur}
	\ \\
	\begin{enumerate}
		\item Lorsque $P' = M'U'$ est un sous-groupe parabolique standard de $G'$ et $\Sigma$ un $L$-paquet de représentations irréductibles de carré intégrables de $M'$, alors les composantes irréductibles des $\Ind_{P'}^{G'}\sigma $, $\sigma \in \Sigma$, forment un $L$-paquet $\Pi$ et la distribution $\Sigma_{\sigma\in\Sigma } \Tr(\Ind_{P'}^{G'}\sigma )$ est stable. Chaque élément $\pi $ de $\Pi $ est composante d'une seule induite $\Ind_{P'}^{G'}\sigma $ et y apparaît avec multiplicité $1$.
		\item Soit $\Pi $ un $L$-paquet de représentations irréductibles tempérées de $G'$. Il existe un sous-groupe parabolique standard $P' = M'U'$ de $G'$ et un $L$-paquet $\Sigma$ de représentations irréductibles de carré intégrable de $M'$ tel que les éléments de $\Pi$ soient les composantes irréductibles des $\Ind_{P'}^{G'}\sigma $, $\sigma\in\Sigma $.

		Alors, il existe une unique représentation irréductible tempérée symplectique $\tau_{\Pi}$ de $G^o$ telle que, pour tout $f\in \hecke(G')$,
		\[
			\sum_{\sigma\in\Sigma } \Tr((\Ind_{P'}^{G'}\sigma) (f^{G'}))= \Tr_{\widetilde{G^0}}(\tau _{\Pi }^+(f))
		\]
		Où $\tau _{\Pi }^+$ est un prolongement à $G$ de $\tau_{\Pi}$ défini en \cite{Arthur} 2.2 (vor également \cite{ArthurEndoscopicClassification} p. 11). Toute représentation irréductible tempérée symplectique de $G^{0}$ est obtenue ainsi.
	\end{enumerate}
\end{thm}
\begin{rem}
	Arthur a en fait montré qu'il y a une partition de l'ensemble des représentations irréductibles de carré intégrable de $M'$ en sous-ensembles finis appelés $L$-paquets et que, pour chaque paquet, la somme des caractères des éléments du paquet est une distribution stable. La propriétés d'induction en est alors déduite à l'aide de la théorie des $R$-groupes (due à Harish-Chandra et Silberger).
\end{rem}
Notons $\distr^{\st}(G')$ l'espace des distributions stables sur $G'$ et $\distr_{\temp}^{\st}(G')$ le sous-espace des distributions qui sont combinaisons linéaires de caractères tempérés stable. Le résultat suivant nous a été communiqué par J.-L. Waldspurger. (Une preuve écrite n'est pour l'instant que disponible dans le cas d'un corps archimédien \cite{WaldspurgerStabilisation3}.)

\begin{thm}[\cite{WaldspurgerStabilisation3} pour le cas archimédien]
	\label{thm:densiteTempereeStable}

	Soit $f\in \hecke(G')$ tel que $D(f) = 0$ pour tout $D \in \distr_{\temp}^{\st}(G')$. Alors, $D(f) = 0$ pour tout $D \in \distr^{\st}(G')$.
\end{thm}

Notons $dm(\pi )$ la mesure de Plancherel sur l'espace $\frak T(G')$ des représentations irréductibles tempérées de $G'$, i.e. la mesure sur $\frak T(G')$ telle que l'on ait, pour tout $f'\in \hecke(G')$, $f'(1) = \int_{\frak T(G')} \Tr(\pi(f')) \, dm(\pi )$.

\begin{cor}[\cite{ShahidiLanglandsConjecture} 9.3]
	\label{cor:plancherelLpaquets}
	La mesure $dm(\pi )$ sur $\frak T(G')$ est constante sur les $L$-paquets.
\end{cor}
\begin{proof}
	On note $\mathcal{M}$ un ensemble de représentants des classes de conjugaisons sous-groupes de Levi de $G'$, si $M \in \mathcal{M}$ on note $E_2(M)$ l'ensemble des classes d'équivalences de représentations irréductibles de $M$ de la série discrète, $\overline{E}_2(M)$ l'ensemble des $L$-paquets de séries discrètes, $d(\sigma)$ désigne le degré formel de $\sigma$ et $\Tr \hat{\sigma}$ est le caractère de $\hat{\sigma} = \Ind_{MN}^{G'} \sigma$. La preuve donnée dans \cite{ShahidiLanglandsConjecture} repose sur la conjecture \cite{ShahidiLanglandsConjecture} 9.2 qui permet d'affirmer l'existence de $\lambda_G(\Sigma) \in \C$ tel que pour tout $f \in \hecke(G)$, on ait
	\[
		f(1) = \sum_{M \in \mathcal{M}} \int_{\Sigma \in \overline{E}_2(M)} \lambda_G(\Sigma) \pars{\sum_{\sigma \in \Sigma} \chi_{\hat{\sigma}}(f)} \, \D \sigma
	\]
	Le théorème \ref{thm:densiteTempereeStable} n'est pas équivalent à cette conjecture mais permet quand même de déduire le résultat en adaptant les arguments de \cite{ShahidiLanglandsConjecture}. Expliquons comment.	
	
	Soit $M \in \mathcal{M}$, et $\Sigma \in \overline{E}_2(M)$ un $L$-paquet de série discrète de $M$. On note
	\[
		\Tr \hat{\Sigma} = \sum_{\sigma \in \Sigma} \Tr \hat{\sigma}
	\] 
	On sait d'après \ref{thm:arthur} que c'est une distribution stable.
	
	Dans le cas où $\Sigma$ est déjà un $L$-paquet de $G'$. La mesure de Plancherel est ici le degré formel. On note $V = \R^{\Sigma}$ et on définit $W$ l'orthogonal de l'espace engendrée par le vecteur $(1)_{\sigma \in \Sigma} \in V$.
	
	Soit $v = (v_{\sigma})_{\sigma \in \Sigma} \in V$, Paley-Wiener pour la trace de Bernstein-Deligne-Kazhdan \cite{BDK}, on peut trouver $f_v \in \hecke(G')$ telle que pour tout représentation irréductible $\pi \notin \Sigma$ on ait $\Tr \pi(f_v) = 0$ et pour $\sigma \in \Sigma$, on ait $\Tr \sigma(f_v) = v_{\sigma}$.
	
	Soit $w \in W$. Si $\chi$ est un caractère tempérée stable associé à un $L$-paquet différent de $\Sigma$ alors $\chi(f_w) = 0$. Par ailleurs, comme $w \in W$ on a alors
	\[
		\sum_{\sigma \in \Sigma} \Tr \hat{\sigma}(f_w) = 0
	\]
	Donc pour toute distribution tempérée stable $D \in \distr_{\temp}^{\st}(G')$, on a $D(f_w) = 0$. En admettant le cas non-archimédien du théorème \ref{thm:densiteTempereeStable}, alors on déduit que pour toute distribution stable $D \in \distr^{\st}(G')$, alors $D(f_w) = 0$. En particulier, la distribution $\OrbInt_{G'}(0, \cdot)$, est stable donc $f_w(0) = 0$. La formule de Plancherel (voir \cite{Waldspurger}) donne donc
	\begin{align*}
		0 = f_w(0) &= \sum_{M \in \mathcal{M}} \int_{\sigma \in E_2(M)} \Tr \hat{\sigma}(f_w) \ dm(\sigma) \\
		&= \sum_{\sigma \in \Sigma} d(\sigma) \Tr \sigma(f_w) \\
		&= \sum_{\sigma \in \Sigma} d(\sigma) w_{\sigma}
	\end{align*}
	Donc le vecteur $(d(\sigma))_{\sigma \in \Sigma} \in V$, est orthogonal à tout $w \in W$, il est donc proportionel au vecteur $(1)_{\sigma \in \Sigma} \in V$, ce qui signifie que le degré formel $d(\sigma)$ est constant pour tout $\sigma \in \Sigma$. Ce qui est bien le résultat voulu.
	
	On considère $V$ l'espace des fonctions $v : \bigcup_{N \in \mathcal{M}} E_2(N) \to \C$ définies sur l'ensemble des classes d'équivalences de séries discrètes de sous-groupes de Levi de $G'$ telles que
	\begin{enumerate}
		\item Si $\tau \in E_2(N)$, alors $v(\tau) = 0$ sauf si $(N, \tau)$ est conjugué à $(M, \sigma_{\nu})$ avec $\sigma \in \Sigma$ et $\nu \in i \mathfrak{a}^*$.
		\item $v$ ne dépend de que la classe de conjugaison de $M$.
		\item L'application $h_{\sigma} : \nu \mapsto h(\sigma_{\nu})$ est régulière.
	\end{enumerate}
	Pour tout élément $v \in V$, Shahidi construit dans \cite{ShahidiLanglandsConjecture} 9.3, une fonction $f_v \in \hecke(G')$, telle que
	\[
		\Tr \hat{\tau}(f_v) = v(\tau)
	\]
	pour toute $\tau$ représentation de la série discrète d'un sous-groupe de Levi de $G'$.
	On considère $W$ la sous-algèbre des fonctions $w \in V$ telles que
	\[
		\sum_{\sigma \in \Sigma} w_{\sigma} = 0
	\]
	Si $w \in W$, alors on vérifie que pour toute distribution tempérée stable $D \in \distr_{\temp}^{\st}(G')$, on a $D(f_w) = 0$. Toujours d'après \ref{thm:densiteTempereeStable}, on déduit que pour toute distribution stable $D \in \distr^{\st}(G')$, alors $D(f_w) = 0$. En particulier, $f_w(0) = 0$. La formule de Plancherel (voir \cite{Waldspurger}) donne donc
	\begin{align*}
		0 = f_w(0) &= \sum_{M \in \mathcal{M}} \int_{\sigma \in E_2(M)} \Tr \hat{\sigma}(f_w) \ dm(\sigma) \\
		&= \sum_{\sigma \in \Sigma} \int_{i \mathfrak{a}^*} w_{\sigma}(\nu) .  m_{\sigma}(\nu) \ d \nu \\
	\end{align*}
	Cela étant valable pour tout $w \in W$, on en déduit que $\sigma \mapsto m_{\sigma}$ est constante sur $\Sigma$ (en effet si $\sigma_1 , \sigma_2 \in \Sigma$ et $b$ une fonction régulière sur $i\mathfrak{a}^*$, on considère l'élément $w \in W$ tel que $w_{\sigma_1} = - w_{\sigma_2} = b$ et $w_{\sigma} = 0$ si $\sigma \ne \sigma_1, \sigma_2$. Alors $\int_{i \mathfrak{a}^*} b(\nu) . (m_{\sigma_1}-m_{\sigma_2})(\nu) \ d \nu = 0$ donc $m_{\sigma_1}-m_{\sigma_2} = 0$ puisque l'égalité est valable pour tout $b$).
\end{proof}

Notons $\Irr(G^{0})^s_{\temp}$ l'ensemble des représentations lisses irréductibles tempérées symplectiques de $G^{0}$. D'après le théorème {\ref{thm:arthur} 2)}, c'est un espace fibré sur $\frak T (G')$. Les fibres sont finies, mais la projection $\frak T (G')$ sur $\Irr(G^{0})^s_{\temp}$ n'est en général pas localement triviale. \`A l'aide de l'identité dans {\ref{thm:arthur} 2)}, on définit une structure topologique sur $\Irr(G^{0})^s_{\temp}$ déduite de la décomposition de $\frak T(G')$ en composante connexe par la théorie de la formule de Plancherel de Harish-Chandra \cite{Waldspurger}. On peut alors munir $\Irr(G^{0})^s_{\temp}$ de la mesure induite par la fibration sur $\Irr(G^{0})^s_{\temp}$, i.e. $dm (\tau _{\Pi }) = \sum _{\pi\in\Pi } dm(\pi )$ pour $\Pi $ un $L$-paquet dans $\frak T (G')$. On constate que la fonction qui associe à un élément $\tau $ de $\Irr(G^{0})^s_{\temp}$ le nombre d'éléments dans sa préimage dans $\frak T (G')$ est presque partout égale à une fonction localement constante.

\begin{thm}
	\label{thm:conjecture}
	Il existe $\lambda > 0$ tel que pour tout $f\in \hecke(G^{0})$, on a

	$$\OrbInt_{\widetilde{G^{0}}}(\eta ,f) = \lambda \int _{\Irr(G^{0})^s_{\temp}} \Tr_{\widetilde{G^0}}(\tau^+(f)) \, dm(\tau ).$$
\end{thm}
\begin{proof}
	En effet, par \ref{prop:transfertIntegraleOrbitale}, le côté gauche est égal à $f^{G'}(1)$. Le théorème {\ref{thm:arthur}} appliqué à l'expression pour $f^{G'}(1)$ donnée par la formule de Plancherel donne alors le c\^oté droit de l'égalité, en utilisant le corollaire \ref{cor:plancherelLpaquets}.
\end{proof}




\bibliographystyle{alpha}
\bibliography{biblio}

\end{document}